\documentclass[11pt]{amsart}
\usepackage{amssymb,amsmath}
\usepackage{graphicx}
\usepackage{amsmath}
\usepackage{amsfonts}
\usepackage{amsthm}
\usepackage{amssymb,latexsym}
\usepackage{fixmath}
\usepackage{mathrsfs,amsbsy}
\usepackage{graphicx}
\usepackage{caption}
\usepackage{subcaption}			% Use pdf, png, jpg, or eps§ with pdflatex; use eps in DVI mode								% TeX will automatically convert eps --> pdf in pdflatex		
\usepackage{booktabs}  %  Need for tables
\usepackage{color}
\usepackage{algorithm,algorithmic}
\usepackage{multirow}
\usepackage{url}
\usepackage{float}
\usepackage{bm}
\usepackage{caption}
\usepackage{makecell}
\usepackage{xr}
\usepackage{mathrsfs,amsbsy}
\usepackage{dsfont}
\usepackage{enumerate}
\usepackage{eucal}
\usepackage{color} %red, green, blue, yellow, cyan, magenta, black, white
\definecolor{mygreen}{RGB}{28,172,0} % color values Red, Green, Blue
\definecolor{mylilas}{RGB}{170,55,241}
\newtheorem{theorem}{Theorem}[section]
\newtheorem{lemma}[theorem]{Lemma}
\newtheorem{remark}[theorem]{Remark}

\newcommand{\rev}[1]{{\color{black}#1}}

\date{\today}
\title[]{ High-Order Enriched Finite Element Methods for \rev{Elliptic Interface} Problems with Discontinuous Solutions}
\author{}
\date{\today}	
\author{ Champike Attanayake$^\ddagger$, So-Hsiang Chou$^\dagger$ and Quanling Deng$^\sharp$}
\thanks{$\dagger$
 Department of Mathematics and Statistics, Bowling Green
State University, Bowling Green, OH, 43403-0221. {{\tt email:chou@bgsu.edu}};\,$\ddagger$ Department of Mathematics,
Miami University
Middletown, OH 45042. {{\tt
e-mail:attanac@muohio.edu}};\,$\sharp$ School of Computing, Australian National University, Canberra, ACT 2601, Australia. {{\tt
e-mail:quanling.deng@anu.edu.au}}}

\begin{document}
\maketitle
\begin{abstract}
\rev{Elliptic interface problems whose solutions are $C^0$ continuous have been well studied over the past two decades.
The well-known numerical methods include the strongly stable generalized finite element method (SGFEM) and immersed FEM (IFEM).
In this paper, we study numerically a larger class of elliptic interface problems where their solutions are discontinuous.
A direct application of these existing methods fails immediately as the approximate solution is in a larger space that covers discontinuous functions.
We propose a class of high-order enriched unfitted FEMs to solve these problems with implicit or Robin-type interface jump conditions. }
We design new enrichment functions that capture the imposed discontinuity of the solution while keeping the condition number from fast growth. A linear enriched method in 1D was recently developed using one enrichment function and we generalized it to an arbitrary degree using two simple discontinuous one-sided enrichment functions. The natural tensor product extension to \rev{the 2D case is} demonstrated. Optimal order convergence in the $L^2$ and broken $H^1$-norms are established.
\rev{We also establish superconvergence at all discretization nodes (including exact nodal values in special cases).}
Numerical examples are provided to confirm the theory. Finally, to prove the efficiency of the method for
practical problems, the enriched linear, quadratic, and cubic
elements are applied to a multi-layer wall model for drug-eluting stents in which zero-flux jump conditions and implicit concentration interface conditions are both present.
\end{abstract}
{\bf \small Key Words.}\keywords{\small{ generalized finite element method, elliptic interface, implicit interface jump condition, Robin interface jump condition, linear and quadratic finite elements.}
%\textbf{Mathematics Subjects Classification}: 65N30, 65N35, 35J05
\maketitle
\section{Introduction}
Consider the interface two-point boundary value problem
\begin{equation}
\begin{cases}
-(\beta(x){u}'(x))'+w(x){u}(x)= f(x), &  x\in (a,\alpha)\cup (\alpha,b),  \label{eqn:First}\\
{u}(a)= {u}(b) = 0,
\end{cases}
\end{equation}
where $w(x)\geq 0$, and  $0<\beta\in C[a,\alpha]\cup C[\alpha,b]$ is discontinuous across the interface $\alpha$
with the jump conditions on ${u}$ and its flux $q:=-\beta {u}'$:
\begin{align}\label{implicit}
[{u}]_{\alpha} &=\lambda F(q^+,q^-),\quad \lambda\in \mathbb R,\, F:[c_1,d_1]^2\to \mathbb R, \\
\label{implicit2}
[\beta {u}']_\alpha &= f_\alpha,  \quad f_\alpha\in\mathbb R,
\end{align}
where the jump quantity
 \[    [s]_\alpha:=s(\alpha^+)-s(\alpha^-), \quad s^{\pm}:=s(\alpha^\pm):=\lim_{\epsilon\to 0^+} s(\alpha\pm \epsilon).\]
The primary variable ${u}$ may stand for the pressure, temperature, or concentration in a medium with certain physical properties
and the derived quantity $q:=-\beta {u}'$ is the corresponding Darcy velocity, heat flux, or concentration flux, which is equally important.
 The piecewise continuous $\beta$ reflects a nonuniform material or medium property (we do not require $\beta$ to be piecewise constant).
 \rev{The function
$w(x)$ reflects the surroundings of the medium characterising the coefficient of the reaction term.}
The case of $\lambda=0$ is widely studied, while the case of $\lambda>0$ gives rise to a more difficult situation. For example, the case of rightward concentration flow \cite{Wang, Zhang1, Zhang2} imposes
\begin{equation}
\begin{cases}\label{jmp1}
[{u}]_\alpha&=\lambda (\beta {u}')(\alpha^-)\\
[\beta {u}']_\alpha &= 0,
\end{cases}
\end{equation}
which generates an implicit condition since the left-sided derivative is unknown. Implicit interface conditions abound in higher dimensional applications \cite{Ammari, Hahn, Kim, Kruit1}.  For definiteness, we will study a class of efficient enriched methods for problem  \eqref{eqn:First} under the jump conditions \eqref{jmp1},

%\color{blue}

The model problem \eqref{eqn:First}, allowing the solution to be discontinuous at the interfaces,
is a more general form of interface problem than the ones studied by Babu\v{s}ka \textit{et. al.} using SGFEM \cite{Babuska2,Babuska1,kergrene2016stable}
and by Li \textit{et.al.} using immersed finite element method (IFEM) \cite{LI,LI:2006,Li2003,Li2004}.
In these works, the interface problem is assumed to have a continuous solution at the interfaces. For the case of the discontinuous solution at the interfaces, some IFE methods
have also been developed, see for example \cite{zhang2019strongly,zhang2020strongly}.
A large class of the methods is developed based on the unfitted meshes, which has been demonstrated to be more efficient than methods using fitted meshes, especially when the interface is moving \cite{he2013immersed}; see also \cite{Bordas,Chou, guo2019group, guo2018nonconforming, jo2019recent, Wang}.

The generalized FEMs (GFEM) were first introduced to capture certain known features of the solution to the crack problem  \cite{babuvska2004generalized, Belyt, Fries, Moes}.
However, it has been shown that GFEM suffers from a lack of robustness with respect to \rev{mesh} configurations and bad conditioning.
\rev{In general, the condition numbers of GFEM can grow with an order $\mathcal{O}(h^{-4})$ where $h$ characterizes the mesh size.
This is two orders of magnitude worse than the standard FEM (known to be of order $\mathcal{O}(h^{-2})$).
To resolve this issue, SGFEM has been developed \cite{Babuska1,Babuska2,kergrene2016stable} with an extra feature of robustness with respect to the mesh configurations.
Further development of SGFEM has been active.  For example,  \cite{zhang2020stable} extends the linear SGFEM to quadratic-order; \cite{Deng} extends the method for eigenvalue interface problems, and \cite{zhang2021generalized} generalizes the SGFEM idea to isogeometric analysis with B-spline basis functions.
%\color{red}
}

\rev{In 1D, when the solution to the underlying interface problem is continuous, a single enrichment function associated with the interface
can be used for arbitrary high-order elements in the SGFEM.
However, for an interface problem whose solution is discontinuous at interfaces,
the natural extension with one enrichment function at each interface fails.
A more sophisticated construction of enrichment functions is desired especially for high-order elements.
This motivates the present work.
}

%%%%%%%%%%%%%%%%%%%%%%%%%%%%%%%%%%%%%%%%%%%%%%%%%%%%%%%%%%%%%%%%%%%%%%%%%%%%%%%%%%%%%%%%%%%%%%%%%%%%
 In our \rev{enriched FEM (might also call it as a GFEM)}, the approximation finite element space $V^{enr}_h$ takes the form:
\begin{equation}\label{struc}
  V^{enr}_{h}:=S_h+ V_E=\{p_{h}+q_{h}\psi:\,p_{h},q_{h}\in S_{h},\psi\in F_{enr}\}
  \end{equation}
where $S_h$ is a standard finite element space (e.g., $\mathbb P_k$-conforming, $k\geq 1$),
\begin{equation}
V_E=\{q_{h}\psi:\,q_{h}\in S_{h},\psi\in F_{enr}\},
\end{equation}
 and the function $\psi$ is from the enrichment function space
\begin{equation}\label{functions}   F_{enr}:=\text{span}\{\psi_0,\psi_1,\ldots,\psi_m\}, \,\, \text{dim} (F_{enr})=m+1.
\end{equation}
Here the basis functions \rev{$\psi_i$ capture the} interface condition(s) at $\alpha$, e.g., zero or nonzero jump of the function value across $\alpha$.
For example, for \rev{a continuous solution case}, a single ($m=0$) enrichment function suffices, whereas we show in this paper that for discontinuous solution case, we need two enrichment functions ($m=1$) defined in \eqref{eqn:lin1}.
There are some distinct features about $V_h^{enr}$ in this case.
Firstly,
the subspaces $S_h$ and $V_E$ have nonempty intersection, i.e., $S_h \cap V_E \ne \emptyset$.
\rev{This nonempty intersection implies linear dependence on the set of functions that consist of the basis functions of $S_h$ and those of $V_E$.
We show that the space $V^{enr}_{h}$ can be characterised by this set of functions when we remove the bubble basis functions associated with interfaces from $S_h$.
}
Secondly, the local shape basis of IFEM utilizes information on discontinuous $\beta$ while GFEM does not.
\rev{We propose an enriched method that} does not require the discontinuous diffusion coefficient to be piecewise constant, which is an advantage. This difference affects the overhead and complexity of the
convergence analysis. To carry out the error analysis in a conforming GFEM, we use the principle that the error in the finite element solution $u_h$ should be bounded by the approximation error in the finite element space $V^{enr}_h$:
\begin{equation}\label{motive}
     ||u-u_h||\leq C \inf_{v \in V^{enr}_h}||u- v ||.
\end{equation}
The optimal error estimate is established by demonstrating the existence of an optimal order approximate piecewise polynomial in $V^{enr}_h$.
%Our method has a simpler analysis since we have an easy candidate in $V^{enr}_h$.
\rev{For the case with $C^0$ continuous solutions,
this approximability was established by constructing an interpolating polynomial; see Deng and Calo \cite{Deng} for details.
With this in mind, the authors demonstrated the optimal convergence of GFEM solutions in {\it all} $\,\mathbb P_p-$conforming spaces ($p\geq 1$) enriched by the well-known hat-like enrichment function (cf. Eq. \eqref{eqn:lin} below). On the other hand, Chou et. al. \cite{ChouAttan2} used a single enrichment
 function to handle problem \eqref{eqn:First} with $[u]_\alpha\neq 0$ using the linear GFEM $(p=1)$.
 In this paper, we use two one-sided enrichment functions for each interface
and show the effectiveness and convergence of the associated GFEM of all orders, i.e., $p\geq 1$.}
In addition to demonstrating optimal order convergence in the $L^2$ and broken $H^1$ norms, we also show superconvergence of nodal errors.

\rev{The organisation of the rest of the paper is as follows.}
In Section 2, we state the weak formulation for the implicit interface condition problem \eqref{eqn:First} with \eqref{jmp1}, and define enrichment functions and spaces.
In Theorem \ref{thm2.6},
we show optimal order convergence in the $L^2$ and broken $H^1$ norms. Furthermore, in Theorem \ref{2order} we show $2p$ order convergence at the nodes and the exactness of the approximate solution at nodes for the piecewise constant diffusion case.
\rev{Section 3 concerns its extension to a special case of 2D problems where the interface is a straight parallel (to a boundary) line.}
In Section \ref{NumExample}, we provide numerical examples of a porous wall model to demonstrate the effectiveness
of the present GFEM and confirm the convergence theory. Linear, quadratic, and cubic elements are tested. Furthermore, following the viewpoint of the SGFEM \cite{Babuska1, Babuska2, Deng}, we compare the condition numbers of our (discontinuous solution) method with those in the continuous solution case \cite{ATTChou2021}, and numerically show that they are comparable for the same mesh sizes.  Finally, in Section 5 we give some concluding remarks.

\section{Enrichment Functions and Spaces}
\subsection{Weak Formulation}
Let $I^-=(a,\alpha)$,$I^+=(\alpha,b)$, $I=(a,b)$, and define
\[  H_{\alpha,0}^1(I)=\{v\in L^2(I): v\in H^1(I^-)\cap H^1(I^+), v(a)=v(b)=0\}.\]
We use conventional Sobolev norm notation. For example,  $|u|_{1,J}$  denotes the usual $H^1$-seminorm for $u\in H^1(J)$, and $||u||^2_{i,I^-\cup I^+}=||u||^2_{i,I^-}+||u||^2_{i, I^+}, i=1,2$ for $u\in  H_\alpha^2(I)$, where
\[
 H_\alpha^2(I)= H^2(I^-)\cap H^2(I^+).\]
 The space $H_{\alpha,0}^1(I)$ is endowed with the $||\cdot||_{1,I^-\cup I^+}$ norm, and $H_\alpha^2(I)$ with the $||\cdot||_{2,I^-\cup I^+}$ norm. The higher-order spaces
$H_\alpha^m(I)$ and $H_{\alpha,0}^m(I), m\geq 3$ are similarly defined.
With this in mind, the weak formulation of the problem \eqref{eqn:First} under (\ref{jmp1}) is: Given $f \in L^{2}(I)$, find ${u}\in H_{\alpha,0}^{1}(I)$ such that
\begin{equation}\label{ExactWeak}
a({u},v) = (f,v)\quad \forall v\in H_{\alpha,0}^{1}(I),
\end{equation}
where
\begin{align*}
a(u,v) =& \int_{a}^{b}\beta(x){u}'(x)v'(x)dx+\int_{a}^{b}w(x){u}(x)v(x)dx+\frac{[{u}]_\alpha[v]_\alpha}{\lambda},\\
(f,v) =& \int_{a}^{b}f(x)v(x)dx.
\end{align*}
 The above weak formulation can be easily derived by integration-by-parts with \eqref{jmp1}. Since $\lambda>0$, the bilinear form $a(\cdot,\cdot)$ is coercive and is bounded due to Poincar$\acute{e}$ inequality. By the Lax-Milgram theorem, a unique solution ${u}$ exists. Throughout the paper, we assume that the functions $\beta$, $f$, and $w$ are such that
 the solution ${u}\in H^p_\alpha(I), p\geq 2$.
  \subsection{Enrichment Functions}\label{sub:enrichment functions}
 We now introduce an approximation space for the solution ${u}$. Let $a=x_{0}<x_{1}<\ldots <x_{k}<x_{k+1}<\ldots<x_{N}=b$ be a partition of $I$ and the interface point
$\alpha\in(x_{k},x_{k+1})$ for some $k$. As usual, the mesh size $h:=\max_i h_i, h_i=x_{i+1}-x_i,i=0,\ldots, N-1$. Define the two one-sided enrichment functions
\begin{equation}
\psi_0(x) :=
\begin{cases}
0   &  x\in [a,x_{k}]\\
m_1(x-x_k)&  x\in [x_{k},\alpha)  \label{eqn:lin1}\\
0   &  x\in (\alpha,b]
\end{cases}
\text{ and }
\psi_1(x) :=
\begin{cases}
0   &  x\in [a,\alpha)\\
m_2(x-x_{k+1}) &  x\in (\alpha,x_{k+1}]  \\
0   &  x\in [x_{k+1},b]
\end{cases}
\end{equation}
where
\begin{equation}\label{slopes}
m_1=\frac{1}{\alpha-x_{k}},\quad m_2=\frac{1}{\alpha-x_{k+1}}.
\end{equation}
%%%%%%%%%%%%%%%%%%%%%%%%%%%%%%%%%%%%%%%%%%%%%%%%%%%%%%%%%%%%%%

{\bf Remark 2.1.}
\begin{itemize}
 \item It is the enrichment function space $F_{enr}:=\text{span}\{\psi_0,\psi_1\}$ that matters. Any two basis functions of $F_{enr}$ qualify, i.e., any pairs of nonzero $m_1$ and $m_2$
 will guarantee convergence, as we shall show below. This space can also be spanned by one continuous and one one-sided function.

 For example, the familiar continuous
\begin{equation}
\tilde \psi_0(x) :=
\begin{cases}
0   &  x\in [a,x_{k}]\\
\tilde m_1(x-x_k)&  x\in [x_{k},\alpha)  \label{eqn:lin10}\\
\tilde m_2(x-x_{k+1})& x\in [\alpha,x_{k+1}]\\
0&   x\in (\alpha, b]
\end{cases}
\text{ and a discontinuous }\quad
\tilde \psi_1(x) :=\psi_1(x)
\end{equation}
where
\begin{equation}\label{slopesC}
\tilde m_1=\dfrac{\alpha-x_{k+1}}{x_{k+1}-x_{k}},\quad   \tilde m_2=\dfrac{(\alpha-x_k)}{(x_{k+1}-x_k)}.
\end{equation}

\item We will adopt choice \eqref{slopes} since it makes the ensuing error analysis more transparent and simpler.
%%%%%%%%%%%%%%%%%%%
\item Note that the slopes $\tilde m_1, \tilde m_2$ in choice \eqref{slopesC} are uniformly bounded by $1$, while the slopes $m_1,m_2$ in choice \eqref{slopes} are not. We show in Appendix \ref{scale_reason} that the bounded type may lead to a system of out-of-scale finite element equations while choice \eqref{slopes} does not. However, after a diagonal scaling, their resulting
preconditioned systems have comparable condition numbers (cf. Section \ref{NumExample}).
\end{itemize}
\color{black}
%%%%%%%%%%%%%%%%%%%%%%%%%%%%%%%%%%%%%%%%%%%%%%%%%%%%%%%%%%%%%%%%%%%%%%%

 Let us describe the enriched space associated with $\psi_i,i=0,1$. Let $\bar I=\cup_0^{N-1}I_{i},I_i=[x_i,x_{i+1}]$ and let $S^p_{h}$ be the standard $\mathbb P_p,p\geq 1$ conforming finite element space
\begin{align}\label{P1}
S^{p}_{h} &= \{
v_{h}\in C(\bar I) : v_{h}|_{I_i} \in \mathbb P_{p}, i=0,\ldots,N-1, v_h(a)=v_h(b)=0\}\\
\notag &=\text{span}\{\phi_j,j\in {\mathcal N}^p_h\},
\end{align}
where $\phi_i$'s are the Lagrange nodal basis functions of order $p$, and where the nodal index set $\mathcal{N}^p_h:=\{1,2,\cdots,pN-1\}$.

We denote the usual $\mathbb P_p$-interpolation operator by ${\mathcal I}_h:C(\bar I)\to S^{p}_h$,
\[     {\mathcal I}^{p}_h g=\sum_{i=1}^{pN-1}g(t_i)\phi_i, \]
where $t_{i}, i\in \mathcal{N}^p_h$ are the nodes such that $\phi_i(t_j)=\delta_{ij}$. For each element $\tau=I_i$, the local interpolation operator ${\mathcal I}_{\tau}^{p}$ is
\begin{equation}
{\mathcal I}^{p}_\tau g:={\mathcal I}^{p}_hg\bigg|_\tau=\sum_{j\in \mathcal N^{p}_{\tau}}g(t_j)\phi_j,
\end{equation}
and $\mathcal N^{p}_{\tau}\subset \mathcal N^{p}_{h}$ is the set of nodes associated with $\tau$.
Define the enriched finite element space
%%%%%%%%%%%%%%%%%%%%%%%%%%%%%%%%%%%%%%%

\begin{align}\label{enriched}
\overline{S}^{p}_{h}&=S^{p}_h+S_{h,E}=\{v_h+w_h,v_h\in S^{p}_h, w_h\in S_{h,E} \}
\end{align}
and
\[      S_{h,E}:=\{w_h: w_h=v\psi_0+w\psi_1,\, v,w,\in S_h^p\}. \]
%where $v\in S^{p}_h$ and $v$ is nonzero over the interface element $[ x_k, x_{k+1} ]$ that contains the interface.
%%%%%%%%%%%%%%%%%%%%%%%%%%%%%%%%%%%%%%%%%%%%%%%%%%%%%%%%%%%%%
It will be shown later that algebraic sum `+' in \eqref{enriched} cannot be a direct sum `$\oplus$' since the intersection space, $S^{p}_h\cap S_{h,E}$, is not empty.

The GFEM for problem (\ref{eqn:First}) under (\ref{jmp1}) is to find $u_{h}\in \overline{S}^p_{h}\subset H^1_{\alpha,0}$  such that
\begin{equation}\label{FEMweak}
a(u_{h},v_{h}) = (f,v_{h})\quad \forall v_{h}\in \overline{S}^p_{h}.
\end{equation}
To derive the corresponding linear algebraic system of equations, we need to form a basis for the enriched space $\overline{S}^p_{h}$, which we characterise as below.

%%%%%%%%%%%%%%%%%%  Interpolation operator%%%%%%%%%%%%%%%%%%%%%%
\subsection{Structure of the enrichment finite element space}
It suffices to analyze the structure of the enriched space restricted to the interface element $[x_k,x_{k+1}]$. To this end, we first concentrate on the reference element $[0,1]$ with an interface point $\hat\alpha$. Let $\xi_i=i/p,i=0,\ldots, p,$ be the evenly distributed nodes and let the corresponding Lagrange interpolating polynomials be
\[   q_i(x)=\prod_{j\neq i} \frac{(x-\xi_j)}{(\xi_i-\xi_j)}.\]
The local enriched space is
\begin{equation}\label{enr}
 V_{h,loc}^{enr}:=V_{loc}+V_E:=\text{span}\{q_0,q_1,\ldots,q_{p-1},q_p\}+\text{span}\{q_i\psi_0,q_i\psi_1\}_{i=0}^{p}.
 \end{equation}
For $p\geq 2,$ we will refer to the $p-1$ functions $\{q_i\}_{i=1}^{p-1}$ as the bubble functions since they vanish at the endpoints. Note that $q_0(0)=1$ and $q_p(1)=1$.
We will show in Lemma \ref{bubble-free1} that the bubble functions belong to $V_E$, and consequently
\[ \text{dim} (V_{h,loc}^{enr}) = \text{dim} (V_{loc}+V_E)=2+2(p+1).\]

\begin{remark}
Note that there is no bubble function for the linear case ($p=1$). Moreover, there holds $V_{loc} \cap V_E = \emptyset.$
Thus, $V_{h,loc}^{enr}:=V_{loc} \oplus V_E$ and the dimension is $\text{dim} (V_{h,loc}^{enr}) =\text{dim} (V_{loc}+V_E)=2+2(p+1) = 6.$
\end{remark}

\begin{lemma}[\bf{Local linear dependence}] \label{bubble-free1}
 Let $p\geq 2$. For any bubble function $q_i,1\leq i\leq p-1$, we have
\begin{equation}
q_i\in V_E:=\text{span}\{q_l\psi_0,q_l\psi_1\}_{l=0}^{p},
\end{equation} i.e., there exist $s_{mj},0\leq m\leq p, j=0,1,$ such that
\[  q_i=\sum_{j=0}^1 \sum_{l=0}^p s_{lj} \psi_j q_l.\]
\end{lemma}
\begin{proof}
We only prove the theorem for $p=2$ since the proof for general $p$ follows closely \rev{in this case} but with more complicated indices. Note that there is only one bubble function
$q_1$ for $p=2$. On $[0,\hat\alpha]$, $\psi_1=0$ and we need to show the existence of $s_{i0}$ such that $q_1=\sum_{i=0}^2s_{i0}\psi_0q_i$. Since $m_1$ can be absorbed into $s_{i0}$, we can assume $m_1=1$ in the following calculation so that $\psi_0=x$.
Expressing $q_i$ as a Taylor polynomial
\[  q_i=q_i(0)+q_i'(0)x+\frac{q''_i(0)}{2}x^2\]
%%%%%%%%%%%%%%%%%%%%%%%%%%%%%%%%%%%%%
%
and comparing the coefficients of $x,x^2,x^3$ of
\begin{equation}
q_1(0)+q'_1(0)x+q''_1(0)x^2=\sum_{i=0}^2s_{i0}\left(q_i(0)x+q'_i(0)x^2+\frac{q''_i(0)}{2}x^3\right)
\end{equation}
we have
\begin{align}\label{rank}
\sum_{i=0}^2s_{i0}q_i(0)&=q'_1(0),\\
\sum_{i=0}^2s_{i0}q_i'(0)&=q''_1(0),\\
\sum_{i=0}^2s_{i0}q''_i(0)&=0,
\end{align}
whose coefficient matrix is the Wronskian matrix evaluated at $x=0$
\begin{equation}\label{wron}
\begin{pmatrix}
q_0(0)&q_1(0)&q_2(0)\\
q'_0(0)&q'_1(0)&q'_2(0)\\
q''_0(0)&q''_1(0)&q''_2(0)
\end{pmatrix}.
\end{equation}

Since the Wronskian is nonsingular, $s_{00},s_{10},s_{20}$ exist.
Similarly for $[\hat \alpha,1]$, we use the Taylor polynomials at $x=1$. All the equations are the same except they are evaluated at $x=1$.
Thus $s_{01},s_{11},s_{21}$ exist. For the general case, the matrix in \eqref{wron} is a Wronskian matrix of order $p$.
\end{proof}
%%%%%%%%%%%%%%%%%%%%%

The approximability of $V_{h,loc}^{enr}$-functions on the reference element is given below
\begin{theorem}[Local approximation] \label{general}
Let $\chi_i,i=1,2$ be the characteristic functions of $[0,\hat\alpha]$ and $[\hat\alpha,1]$, respectively. Define the space
\begin{equation}\label{piecewise}
W=\{w: w=w_1\chi_1+w_2\chi_2, w_i\in \mathbb P_p,i=1,2\}.
\end{equation}
Then
\[  W\subset V_{h,loc}^{enr}.\]
In other words, given two polynomials
 \[     u_1=\sum_{l=0}^pa_l x^l \text{   and   }\quad  u_2=\sum_{l=0}^pb_l(x-1)^l,\]
there exist unique $\zeta_0,\zeta_p;\{\zeta_i\}_{i=p+1}^{2p+1},\{\zeta_j\}_{j=2p+2}^{3p+2}$ such that
  \begin{equation}\label{equate}
\sum_{i=1}^2\chi_iu_i=\sum_{j=0,p}\zeta_jq_j+\sum_{j=0}^{p}\zeta_{j+p+1}q_j\psi_0+\sum_{j=0}^p\zeta_{j+2p+2}
q_j\psi_1.
\end{equation}
\end{theorem}
\begin{proof}
To prove the uniqueness of the solution of the square system in the unknown $\zeta_i$'s, we proceed as follows:
 \begin{align}\label{p+1power}
\sum_{i=0,p}\zeta_iq_i+\sum_{i=0}^p\zeta_{i+(p+1)}q_i\psi_0&=\sum_{l=0}^pa_lx^l,\\
\sum_{i=0,p}\zeta_iq_i+\sum_{i=0}^p\zeta_{i+2(p+1)}q_i\psi_1&=\sum_{l=0}^pb_l(x-1)^l.
 \end{align}
We can take $m_1=m_2=1$ because they can be absorbed into coefficients.
Switching the indices in
\[\sum_{i=0,p}\zeta_i\left(\sum_{l=0}^p\frac{q^{(l)}_i(0)}{l!}x^l\right)+\sum_{i=0}^p\zeta_{i+(p+1)}
 \left(\sum_{l=0}^p\frac{q^{(l)}_i(0)}{l!}x^l\right)x=\sum_{l=0}^pa_lx^l,\]
 we obtain
 \begin{align}\label{set1}
\sum_{l=0}^p\left(\sum_{i=0,p}\zeta_i\frac{q^{(l)}_i(0)}{l!}\right)x^l+
 \sum_{l=1}^{p+1}\left(\sum_{i=0}^p\zeta_{i+(p+1)}\frac{q^{(l-1)}_i(0)}{(l-1)!}\right)x^l=\sum_{l=0}^pa_ix^l.
 \end{align}
 Similarly from
 \[
\sum_{i=0,p}\zeta_i\left(\sum_{l=0}^p\frac{q^{(l)}_i(1)}{l!}(x-1)^l\right)+\sum_{i=0}^p\zeta_{i+2(p+1)}
 \left(\sum_{l=0}^p\frac{q^{(l)}_i(1)}{l!}(x-1)^l\right)(x-1)=\sum_{l=0}^pb_l(x-1)^l,
\]
we have
 \begin{align}\label{set2}
\sum_{l=0}^p\left(\sum_{i=0,p}\zeta_i\frac{q^{(l)}_i(1)}{l!}\right)(x-1)^l+
 \sum_{l=1}^{p+1}\left(\sum_{i=0}^p\zeta_{i+2(p+1)}\frac{q^{(l-1)}_i(1)}{(l-1)!}\right)(x-1)^l=\sum_{l=0}^pb_l(x-1)^l.
 \end{align}

 Comparing the coefficients of $x^{p+1}$ from the two sides of \eqref{set1} and \eqref{set2}, we see that
 \begin{align}\label{highest}
  \sum_{i=0}^p\zeta_{i+(p+1)}q^{(p)}_i(0)&=0,\\
  \sum_{i=0}^p\zeta_{i+2(p+1)} q^{(p)}_i(1)&=0.
 \end{align}
 Comparing the coefficients of $x^l,0\leq l\leq p$, we arrive at
 \begin{align*}
 \sum_{i=0,p}\zeta_i\frac{q^{(l)}_i(0)}{l!}+\sum_{i=0}^p\zeta_{i+(p+1)}\frac{q^{(l-1)}_i(0)}{(l-1)!}&=a_l,\\
  \sum_{i=0,p}\zeta_i\frac{q^{(l)}_i(1)}{l!}+\sum_{i=0}^p\zeta_{i+2(p+1)}\frac{q^{(l-1)}_i(1)}{(l-1)!}&=b_l,
 \end{align*}
 with the understanding that the second terms on the left drop out when $l=0$.
 Collecting terms, we see that the square system is
 \begin{align}\label{ell1}
    \sum_{i=0}^p\zeta_{i+(p+1)}q^{(p)}_i(0)&=0,\\
  \label{ell2}   \sum_{i=0,p}\zeta_i\frac{q^{(l)}_i(0)}{l!}+\sum_{i=0}^p\zeta_{i+(p+1)}\frac{q^{(l-1)}_i(0)}{(l-1)!}&=a_l,\, 0\leq l\leq p,\\
 \label{ell3}     \sum_{i=0}^p\zeta_{i+2(p+1)} q^{(p)}_i(1)&=0,\\
  \label{ell4} \sum_{i=0,p}\zeta_i\frac{q^{(l)}_i(1)}{l!}+\sum_{i=0}^p\zeta_{i+2(p+1)}\frac{q^{(l-1)}_i(1)}{(l-1)!}&=b_l,\,0\leq l\leq p.
 \end{align}
 Since this system is square it suffices to show that all $\zeta_i$'s are zero when $a_l=b_l=0\leq l\leq p$.
 First, observe that $\zeta_0=\zeta_p=0$. In fact,  with $l=0$ using the second and fourth equations we have
 \begin{align*}
  \zeta_0q_0(0)+\zeta_pq_p(0)=0,\\
  \zeta_0q_0(1)+\zeta_pq_p(1)=0.
  \end{align*}
  The conclusion follows using $q_p(1)=1,q_p(0)=0, q_0(0)=1,q_0(1)=0$.
  As a consequence, \eqref{ell2} and \eqref{ell4} decouple and simplify to two linear systems
  \begin{align}\label{final1}
  \sum_{i=0}^p\zeta_{i+(p+1)}q^{(l-1)}_i(0) &=0,\\
 \label{final2} \sum_{i=0}^p\zeta_{i+2(p+1)}q^{(l-1)}_i(1)&=0
 \end{align}
 whose coefficient matrices are nonsingular, being Wronskian matrices. Thus, all $\zeta_i=0$.
\end{proof}

Transforming the above results on [0,1]  to $[x_k,x_{k+1}]$, we have the following result.
\begin{lemma}[Local interpolant] \label{local}
 Let $g\in C(\bar \omega), \omega=(c,d)$ and  $I^p_{h,\omega}g\in \mathbb P_p$ be the interpolating polynomial of $g$ at $p+1$ nodes $\xi_i=c+i[d-c]/p,i=0,\ldots, p$. Then for $v\in C[x_{k},\alpha]\cap C[\alpha,x_{k+1}]$, there exists a $\rho \in V_h^{enr}\bigg|_{[x_k,x_{k+1}]}$ such that
\begin{equation}
\left(I^p_{h,[x_k,\alpha]}v\right)\chi_{[x_k,\alpha]}+\left(I^p_{h,[\alpha,x_{k+1}]}v\right)\chi_{[\alpha,x_{k+1}]}=\rho.
\end{equation}
\end{lemma}
%%%%%%%%%%%%%%%%%%%
Combining the classical results for non-interface elements and Lemma \ref{local} for the interface element, we have
\begin{lemma}[Global interpolant] \label{interpo}
Define a global interpolant $\mathcal I_{h,E}: C[a,\alpha]\cap C([\alpha,b])\to \mathbb R$ by
$\mathcal I_{h,E}v = I^p_{h,\omega}v, \omega=[x_{k},\alpha],[\alpha,x_{k+1}], [x_i,x_{i+1}], i=0,\ldots,k-1, k+1,\ldots,N-1.$
Then
\[  |v-\mathcal I_{h,E}v|_{1,I^-\cup I^+}\leq Ch^p|v|_{p+1,I^-\cup I^+}. \]
\end{lemma}
\begin{theorem}[Error Estimate]  \label{thm2.6}
Let ${u}$ be the exact solution  and ${u}_{h}$ be the approximate solution of (\ref{ExactWeak})  and (\ref{FEMweak}), respectively. Then
there exists a constant $C>0$ such that
\begin{equation}\label{main}
\|{u}-{u}_{h}\|_{0,I^-\cup I^+} + h\|{u}-{u}_{h}\|_{1,I^-\cup I^+} \leq Ch^{p+1}\|{u}\|_{p+1,I^{-}\cup I^{+}}
\end{equation}
provided that the norm of the exact solution on the right side is finite.
The constant $C$ is independent of $h$ and $\alpha$ but depends on the ratio  $\rho:=\frac{\beta^{*}}{\beta_{*}}$ with $\beta^{*} = \sup_{x\in [a,b]}\beta(x)$ and
$\beta_{*} = \inf_{x\in [a,b]}\beta(x)$.
\end{theorem}
\begin{proof}
Subtracting (\ref{ExactWeak}) from (\ref{FEMweak}), we have
\[
a({u}-{u}_{h},q_{h}) = 0\quad\forall q_{h}\in\overline{S}_{h}.
\]
Using the boundedness and coercivity properties of the bilinear form $a(\cdot,\cdot)$, we get
\begin{eqnarray*}
\beta_{*}|{u}-{u}_{h}|^2_{1,I} &\leq & a({u}-{u}_{h},{u}-{u}_{h})  = a({u}-{u}_{h},{u}-q_{h}) \\
&\leq& \beta^{*}|{u}-{u}_{h}|_{1,I}  |{u}-q_{h}|_{1,I},
\end{eqnarray*}
where $\beta^{*} = \sup_{x\in [a,b]}\beta(x)$  and $\beta_{*} = \inf_{x\in [a,b]}\beta(x)$.  Thus, by Cea's lemma and Lemma \ref{interpo}
\begin{eqnarray*}
|{u}-{u}_{h}|_{1,I}   &\leq& \frac{\beta^{*}}{\beta_{*}}\inf |{u}-q_{h}|_{1,I}  \\
&\leq& \frac{\beta^{*}}{\beta_{*}} |{u}- \mathcal{I}_{h,E} u|_{1,I} \\
&\leq& Ch^p\|{u}\|_{p+1,I^{-}\cup I^{+}}.
\end{eqnarray*}
Then the usual duality argument leads to
\[
\|{u}-{u}_{h}\|_{0,I}  \leq  Ch^{p+1}\|{u}\|_{p+1,I^{-}\cup I^{+}}.
\]
\end{proof}

We note that the jump ratios  $\rho:=\frac{\beta^{*}}{\beta_{*}}$ are of moderate size for the wall model in Section \ref{NumExample}.

%%%%%%%%%%%%%%%%%%%%%%%%%%%%%%%%%%%%%%%%%%%%%%%%%%%%%%%%%%%%%%%%%
\begin{theorem}[\bf{$2p$-th order accuracy at nodes}] \label{2order}
Suppose that $\beta\in C^{p+1}(a,\alpha)\cap C^{p+1}(\alpha,b),p\geq 1$ and $w(x)=0$. Let ${u}$ be the exact solution  and ${u}_{h}$ be the approximate solution of (\ref{ExactWeak})  and (\ref{FEMweak}), respectively.
Then there exists a constant $C>0$ such that
\begin{equation}\label{2nd}
  |{u}(\xi)-{u}_h(\xi)|\leq Ch^{2{p}}||{u}||_{p+1,I^-\cup I^+},\quad \xi=x_i,1\leq i\leq n-1,
  \end{equation}
where $C$ depends on certain norms of Green's function at $\xi$.\\
{\bf Superconvergence}. Furthermore, if $\beta$ is piecewise constant with respect to $[a,\alpha]$ and $[\alpha,b]$, then
\begin{equation}\label{exact}
  u(\xi)=u_h(\xi) \quad \forall \xi=x_i,1\leq i\leq n-1.
  \end{equation}

\end{theorem}
\begin{proof}
Let $g=G(\cdot,\xi),\xi\ne \alpha$ be the Green's function satisfying
\[ a(G(\cdot,\xi),v)=<\delta(x-\xi),v>,\quad v\in H^1_{0,\alpha}(a,b)\]
whose existence is guaranteed by the Lax-Milgram theorem, since in 1D point evaluation is a bounded operator.
We can find Green's function via the classical formulation (for simplicity let $[a,b]=[0,1]$ and $\xi<\alpha$):
\begin{align*}   -(\beta g')'&=\delta(x-\xi), 0<x<1,  &g(0)=g(1)=0,\\
 [g]_\alpha&=\gamma[g']_\alpha,  &[\beta g']_\alpha=0,\\
 [g]_\xi&=0,  &[\beta g']_\xi=1.
 \end{align*}
 Define
\begin{equation}\label{Kfunction}
   K(x):=\int_0^x\frac{1}{\beta(t)}dt,\,0\leq x\leq\alpha;\quad K^c(x):=\int_x^1\frac{1}{\beta(t)}dt,\, \alpha\leq x\leq 1.
   \end{equation}
 Then, similar to the techniques in \cite{Chou} we have
 \begin{equation}
 G(x,\xi)=
 \begin{cases}\label{Green}
 c_1K(x),\qquad &0\leq x\leq \xi\\
 c_3(K(x)-K(\xi))+c_1K(\xi),\qquad &\xi\leq x< \alpha \\
 -c_2K^c(x),\qquad &\alpha< x\leq 1,
 \end{cases}
 \end{equation}
 where with $\gamma=\frac{\lambda\beta^+\beta^-}{(\beta^--\beta^+)}$
 \[  c_3=\dfrac{K(\xi)}{-K^c(\alpha)-K(\alpha)-\lambda},\,c_2=c_3,\,c_1=1+c_3.\]

Thus, for $\xi< \alpha$, $g=G(\cdot,\xi)\in H^{p+1}(\Omega)$, for $\Omega=(a,\xi),(\xi,x_k),(x_k,\alpha),(\alpha,x_{k+1})$, and $(x_{k+1},b)$. Similar regularity holds if $\xi>\alpha$. Using the local estimates in Lemma \ref{interpo}, we conclude that
that there exists $ I_h g\in \bar S_h^p$ such that

\begin{equation}\label{gr} |g- I_h g|_{1,\Omega}\le C h^p ||g^{(p+1)}||_{0,\Omega}
\end{equation}
for all the $\Omega$'s listed above.
Now with $\xi=x_i$
\[   e(x_i)=a(g, e)=a(g- I_h g,e)\] implies that
\begin{align*}
 |e(x_i)|&\le Ch^p||g||_{p+1,*}\,h^p||u||_{p+1,I^-\cup I^+}\\
  &\leq Ch^{2p}||g||_{p+1,*}||u||_{p+1,I^-\cup I^+}
  \end{align*}
   where $||g||^2_{p+1,*}:=\sum||g||^2_{p+1,\Omega},$ the summation being over all the $\Omega$'s listed above.

   We next prove the assertion \eqref{exact}. Since $\beta$ is piecewise constant, Green's function \eqref{Green} is  piecewise linear in $x$ with fixed $\xi_i=x_i,1\leq i\leq n-1$. Hence, $g=G(\cdot,x_i)$ is in $\bar S_h^p$ by Theorem \ref{general} and consequently
   \[e(x_i)=a(g, e)=0.\]
   \end{proof}
Let us point out in passing that the assertion \eqref{2nd} still holds if $0\leq w\in C^{p}[a,b]$,  which can be proven by the techniques of the reduction theorem (p. 4, \cite{wahlbin2006superconvergence}) to the divergence form.
%%%%%%%%%%%%%%%%%%%%%%%%%%%%%%%%%%%%%%%%%%%%%%%%%%%%%%%%%%%%%%%%%%%%%%%%%%%%%%
 %%%%%%%%%%%%%%%%%%%%%%%%%%%%%%%%%%%%%%%%%%%%%%%%%%%%%%%%%%%%%%%%%%%%%%%%
%%%%%%%%%%%%%%%%%%%%%%%%%%%%%%%%%%%%%%%%%%%%%%%%%%%%%%%%%%%%%%%%%%%%%%%%%%%%%%%%%%%%%%%%%%%%%%%%%%%

We will use the preconditioned conjugate gradient method to solve the resulting finite element equation.
 %In Appendix \ref{scale_reason}, we show that the stiffness matrix $A$ is not well-scaled when the slopes in choice \eqref{slopes} take the bounded form of %$m_1=(\alpha-x_{k+1})/h_k,m_2=(\alpha-x_k/h_k)$ instead of the unbounded form of $m_1=1/(\alpha-x_k),m_2=1/(\alpha-x_{k+1})$ by examining the its diagonal entries. In practice,
  The matrix $D_A$, the diagonal part of $A$, is used as a preconditioner, and we solve iteratively a system of the form $D_A^{-1/2}AD_A^{-1/2}y=b$. The scaled condition number $\kappa_2(D_A^{-1}A)$ (SCN)\cite{Babuska2} for the higher order method will be computed in Section \ref{NumExample}.

\section{Extension to Two Dimensions}

Let $\Omega:=[a,b]\times [c,d]$ be a rectangular domain
with a vertical or horizontal interface $\Gamma:=\{(x,y):x=\alpha \}\cap \bar \Omega$ or $\Gamma:=\{(x,y):y= \alpha \}\cap \bar \Omega$
that separates the domain into two subdomains $\Omega^-$ and $\Omega^+$.
Then, $\Gamma = \bar \Omega^- \cap \bar \Omega^+$ with $\Omega^- \cap \Omega^+ = \emptyset$ and $\bar \Omega^- \cup \bar \Omega^+ = \bar \Omega$.
We now consider the following interface elliptic problem:
Find $u: \bar \Omega\to \mathbb R$ such that
\begin{equation} \label{eq:pde2d}
\begin{cases}
-\nabla\cdot(\beta({\bf x})\nabla {u})+w({\bf x}){u}({\bf x}) = f({\bf x}),  & {\bf x}\in \Omega\backslash \Gamma, \\
\qquad\qquad\qquad\qquad\qquad u({\bf x}) =0, & {{\bf x} \in \partial \Omega},
\end{cases}
\end{equation}
where $w({\bf x})\geq 0$, and  $0<\beta\in C(\bar \Omega^-)\cup C(\bar \Omega^+)$ is discontinuous across the interface $\Gamma$.
%which separates $\Omega$ into a left piece $\Omega^-$ and a right piece $\Omega^+$.
{As in 1D, we impose the 2D jump} conditions on ${u}$ and its flux $q:=-\beta \nabla u$ across $\Gamma$:
%\begin{align}\label{implicit2D}
%[{u}]_{\alpha} &=\lambda F(q^+,q^-,[\nabla u\cdot {\bf n}]_\alpha),\quad \lambda\in \mathbb R,\, F:[c,d]^3\to \mathbb R, \\
%\label{implicit22D}
%[\beta \nabla{u}\cdot {\bf n}]_\alpha &= f_\alpha,  \quad f_\alpha\in\mathbb R,
%\end{align}
 \begin{equation}
\begin{cases}\label{jmp12D}
[{u}]_\alpha& = \lambda (\beta \nabla {u}\cdot {\bf n})(\alpha^-), \\
[\beta \nabla u\cdot {\bf n}]_\alpha &= 0,
\end{cases}
\end{equation}
where the jump quantity $[\cdot]$ is understood similarly.
\rev{We consider this interface problem to demonstrate our method for simplicity.
Essentially, this work generalizes the SGFEM developed in \cite{Babuska2,kergrene2016stable} for the straight interface problem \eqref{eq:pde2d}-\eqref{jmp12D} with $\lambda =0$ to the case with $\lambda \ne 0$.}
% \[    [s]_\alpha:=s(\alpha^+)-s(\alpha^-), \quad s({\alpha^\pm}):=s|_{\Omega^{\pm}}(\alpha).\]
%
%\subsection{Enrichment Functions and Spaces in 2D}
\subsection{Enrichment Functions and Weak Formulation in 2D}
Define
\[  H_{\alpha,0}^1(\Omega)=\{v\in L^2(\Omega): v\in H^1(\Omega^-)\cap H^1(\Omega^+), v=0 \text{ on } \partial \Omega\}.\]
We use conventional Sobolev norm notation. For example,  $|u|_{1,J}$  denotes the usual $H^1$-seminorm for $u\in H^1(J)$, and $||u||^2_{i,\Omega^-\cup \Omega^+}=||u||^2_{i,\Omega^-}+||u||^2_{i, \Omega^+}, i=1,2$ for $u\in  H_\alpha^2(\Omega)$, where
\[
 H_\alpha^2(\Omega)= H^2(\Omega^-)\cap H^2(\Omega^+).\]
 The space $H_{\alpha,0}^1(\Omega)$ is endowed with the $||\cdot||_{1,\Omega^-\cup \Omega^+}$ norm, and $H_\alpha^2(\Omega)$ with the $||\cdot||_{2,\Omega^-\cup \Omega^+}$ norm. The higher-order spaces
$H_\alpha^m(\Omega)$ and $H_{\alpha,0}^m(\Omega), m\geq 3$ are similarly defined.
With this in mind, the weak formulation of the problem \eqref{eq:pde2d} under \eqref{jmp12D}
is: Given $f \in L^{2}(\Omega)$, find ${u}\in H_{\alpha,0}^{1}(\Omega)$ such that
\begin{equation}\label{ExactWeak2D}
a({u},v) = (f,v)\quad \forall v\in H_{\alpha,0}^{1}(\Omega),
\end{equation}
where
\begin{align*}
a(u,v) =& \int_\Omega\beta({\bf x})\nabla u({\bf x})\cdot \nabla v({\bf x})d{\bf x}+\int_\Omega w({\bf x}){u}({\bf x})v({\bf x})d{\bf x}+\frac 1 \lambda \int_{{\Gamma}}[{u}]_\alpha[v]_\alpha d\sigma,\\
(f,v) =& \int_\Omega f({\bf x})v({\bf x})d{\bf x}.
\end{align*}

 As in 1D, the above weak formulation can be easily derived by integration-by-parts with \eqref{jmp12D}. With $\lambda>0$, the bilinear form $a(\cdot,\cdot)$ is coercive due to Poincar$\acute{e}$ inequality (when $w=0$ ) and is bounded. By the Lax-Milgram theorem, a unique solution ${u}$ exists. Throughout the paper, we assume that the functions $\beta$, $f$, and $w$ are such that
 the solution ${u}\in H^p_\alpha(\Omega), p\geq 2$.

We assume from now on that the interface is $x=\alpha$ and define a regular rectangular mesh on $\Omega$ as follows. Let $a=x_{0}<x_{1}<\ldots <x_{k}<x_{k+1}<\ldots<x_{N}=b$ be a partition of $I=(a,b)$
{with $\alpha \in (x_{k},x_{k+1})$ for some $k$.}
%and $\alpha$, the $x-$coordinate of the interface, is in $(x_{k},x_{k+1})$ for some $k$.
Similarly, let $c=y_{0}<y_{1}<\ldots<y_{M}=b$ be a partition of $J=(c,d)$.
  Let
\[I_i=[x_i,x_{i+1}], \,  J_j=[y_j,y_{j+1}], \, R_{i,j}=I_i\times J_j,\quad  0\leq i\leq N_x-1,0\leq j\leq N_y-1\]
then $R_{i,j}$ is an element and $R_h=\{R_{i,j}\}$ is the rectangular mesh for $\Omega$. For each $R=R_{i,j}$, we let $h(R)=\max\{h_1,h_2\}, \rho(R)=\min\{h_1,h_2\}$ where $h_1=|x_i-x_{i+1}|, h_2=|y_j-y_{j+1}|$.
We assume $R_h$ is regular, i.e.,  there exists a constant $C>0$ such that
\[ \max_{R\in R_h}\frac{h(R)}{\rho(R)}\leq C.\]
Furthermore, the intersection of the interface $x=\alpha$ with the interface element boundary $\partial R_i$ cannot be a vertex or have nonzero length.

Recall from  \eqref{P1} that $S_{h}$ is the standard $\mathbb P_p,p\geq 1$ conforming finite element space on $I$.
Similarly, we define the standard $\mathbb P_p,p\geq 1$ conforming finite element space on $J$
\begin{align}\label{P1_2D}
{\mathcal S}^{p}_{h} &= \{
v_{h}\in C(\bar  J) : w_{h}|_{J_i} \in \mathbb P_{p}, i=0,\ldots,N_y-1, w_h(c)=w_h(d)=0\}\\
\notag &=\text{span}\{\chi_j,j\in {\mathcal M}^p_h\},
\end{align}
where $\chi_i$'s are the Lagrange nodal basis functions of order $p$ with the nodal index set $\mathcal{M}^p_h:=\{1,2,\cdots,pN_y-1\}$.

It is then nature to define the 2D enriched finite element space ${\mathcal E}^p_h$ as the tensor product of the 1D enriched space $\overline{S}^{p}_{h}$ of \eqref{enriched}
and ${\mathcal S}^{p}_{h}$, the 1D conforming space of \eqref{P1_2D} in the $y-$direction:
\begin{equation}\label{enriched2D}
{\mathcal E}^p_h=\overline{S}^{p}_{h}\otimes{\mathcal S}^{p}_{h}.
\end{equation}

The GFEM or the enriched FEM for problem (\ref{eq:pde2d}) under (\ref{jmp12D}) is to find $u_{h}\in {\mathcal E}^p_{h}\subset H^1_{\alpha,0}(\Omega)$  such that
\begin{equation}\label{FEMweak2D}
a(u_{h},v_{h}) = (f,v_{h})\quad \forall v_{h}\in {\mathcal E}^p_{h}.
\end{equation}

The bilinear form $a(\cdot, \cdot)$ has three parts.
When $\beta|_{\Omega^-}, \beta|_{\Omega^+}, w|_{\Omega^+}$ and $w|_{\Omega^+}$ are constant,
one can derive the 2D matrices as tensor-products of 1D matrices.
With this in mind, for two basis functions $\phi_{i_x}^p(x) \phi_{i_y}^p(y), \phi_{j_x}^p(x) \phi_{j_y}^p(y) \in {\mathcal E}^p_{h}$, the matrix entries are
\begin{equation} \label{eq:sepa2d}
(w \phi_{i_x}^p(x) \phi_{i_y}^p(y), \phi_{j_x}^p(x) \phi_{j_y}^p(y)) = (w\phi_{i_x}^p(x), \phi_{j_x}^p(x) ) \cdot (\phi_{i_y}^p(y),  \phi_{j_y}^p(y)),
\end{equation}
where
\begin{equation}
\begin{aligned}
(w\phi_{i_x}^p(x), \phi_{j_x}^p(x) ) & = \int_a^b w\phi_{i_x}^p(x) \phi_{j_x}^p(x) \ \text{d} x, \\
(\phi_{i_y}^p(y),  \phi_{j_y}^p(y)) & = \int_c^d \phi_{i_y}^p(y)  \phi_{j_y}^p(y) \ \text{d} y.
\end{aligned}
\end{equation}
Similarly,
\begin{equation} \label{eq:sepb2d}
\begin{aligned}
(\beta \nabla \phi_{i_x}^p(x) \phi_{i_y}^p(y), \nabla \phi_{j_x}^p(x) \phi_{j_y}^p(y)) & = (\beta \partial_x \phi_{i_x}^p(x), \partial_x \phi_{j_x}^p(x) ) \cdot (\phi_{i_y}^p(y),  \phi_{j_y}^p(y)), \\
& \quad + (\beta \phi_{i_x}^p(x), \phi_{j_x}^p(x) ) \cdot (\partial_y \phi_{i_y}^p(y), \partial_y \phi_{j_y}^p(y)), \\
\int_{\Gamma} [\phi_{i_x}^p(x) \phi_{i_y}^p(y)]_\alpha [\phi_{j_x}^p(x) \phi_{j_y}^p(y)]_\alpha d\sigma & = \int_c^d [\phi_{i_x}^p(x) \phi_{i_y}^p(y)]_\alpha [\phi_{j_x}^p(x) \phi_{j_y}^p(y)]_\alpha d y \\
& = ([\phi_{i_x}^p(x)]_\alpha, [\phi_{j_x}^p(x)]_\alpha) \cdot (\phi_{i_y}^p(y),  \phi_{j_y}^p(y)).
\end{aligned}
\end{equation}
where $([\phi_{i_x}^p(x)]_\alpha, [\phi_{j_x}^p(x)]_\alpha) = [\phi_{i_x}^p(x)]_\alpha \cdot [\phi_{j_x}^p(x)]_\alpha$.

Using this property of the discretization, the GFEM \eqref{FEMweak2D} leads to the matrix problem of the form
\begin{equation} \label{eq:sepmp2d}
(K_{x,\beta} \otimes M_y + M_{x,\beta} \otimes K_y) U + (M_{x,{w}} \otimes M_y) U + (J_{x,\lambda} \otimes M_y)U = F,
\end{equation}
\rev{where the 1D matrices can be constructed from the inner products above and where $U\in \mathbb R^{(p(N_x+1)+2)\cdot(pN_y-1)}$ is a stack of column vectors $U_{i,y}\in \mathbb R^{pN_y-1}, 1\leq i\leq p(N_x+1)+2$. }
%(In the code $N=(n_x-1)/p$ and the dimension for $p=2$ is thus $n_x+3$.)

\subsection{The approximability of the enriched space}
Let $\hat R=[0,1]\times [0,1]$ be the reference element of the tensor mesh of $\Omega=\cup_k R_k$, i.e, there exists a {\it unique} bijective affine transformation $F_k$ such that $F_k(\hat R)=R_k$ for every $k$.  The interface $x=\hat\alpha$ separates $\hat R$ into $\hat R^+$ and $\hat R^-$.
Then the 2D counterpart of Theorem \ref{general} still holds.
\begin{theorem}[Local approximation] \label{general2D}
Let $\chi_i,i=1,2$ be the characteristic functions of $\hat R^-$ and $\hat R^+$, respectively. Define the space
\begin{equation}\label{piecewise2D}
W=\{w: w=w_1\chi_1+w_2\chi_2, w_i\in \mathbb Q_p:=\mathbb P_p\otimes\mathbb P_p,i=1,2\}.
\end{equation}
\rev{Then
\[  W\subset {\mathcal V}_{h,loc}^{enr},\]
where ${\mathcal V}_{h,loc}^{enr}$ is the local enriched space, ${\mathcal V}_{h,loc}^{enr}=F_k^{-1}({\mathcal E}^p_h|_{R^k})$.}
\end{theorem}
\begin{proof} First note that each $F^{-1}_k$ sends the interface $x=\alpha$ in an interface element $R_k$ to an interface $\hat x=\hat\alpha$ in $\hat R$.
It suffices to consider $w_i$ having the special form of basis functions
\[     w_i(x,y)=X_i(x)Y_i(y), X_i, Y_i\in \mathbb P_p, i=1,2.\]
Then using Theorem \ref{general} to conclude  $X_i$ belongs to the $\bar S_h^p$-space on $[0,1]$ and the fact that $F_k^{-1}$ preserves the tensor product structure of \eqref{enriched2D} complete the proof.
\end{proof}
Using the above theorem, it is not hard to see that the 2D version of Lemma \ref{local} and Lemma \ref{interpo} holds respectively due to the optimal order of approximation
of $\mathbb Q_p(\hat R)$ to $H^2(\hat R)$ functions under the regular mesh assumption.
Consequently, the 2D version of Theorem \ref{thm2.6} holds and we have

\begin{theorem}[Error estimate]  \label{thm2.62D}
Let ${u}$ be the exact solution  and ${u}_{h}$ be the approximate solution of (\ref{ExactWeak2D})  and (\ref{FEMweak2D}), respectively. Then
there exists a constant $C>0$ such that
\begin{equation}\label{main2D}
\|{u}-{u}_{h}\|_{0,\Omega^-\cup \Omega^+} + h\|{u}-{u}_{h}\|_{1,\Omega^-\cup \Omega^+} \leq Ch^{p+1}\|{u}\|_{p+1,\Omega^{-}\cup \Omega^{+}}
\end{equation}
provided that the norm of the exact solution on the right side is finite.
The constant $C$ is independent of $h$ and $\alpha$ but depends on the ratio  $\rho:=\frac{\beta^{*}}{\beta_{*}}$ with $\beta^{*} = \sup_{{\bf x}\in \bar \Omega}\beta({\bf x})$ and
$\beta_{*} = \inf_{{\bf x}\in \bar\Omega}\beta({\bf x})$.
\end{theorem}

\section{Numerical Examples}\label{NumExample}
In this section, we present numerical examples to confirm the theoretical findings. In subsection \ref{confirm} we use Problem 4.1 to verify for the 1D case optimal order convergence \eqref{2nd} and nodal exactness \eqref{exact}
in Theorem \ref{2order} using linear, quadratic, and cubic elements. For the 2D tensor case, we use Problem 4.2 to verify optimal order convergence rates in
Theorem \ref{thm2.62D}. In subsection \ref{quacubic}, we test our methods on a much more complicated
physical example of the multi-layer porous wall model for the drug-eluting stents \cite{Pontrelli} that has been studied using IFEM \cite{Wang, Zhang1, Zhang2,  Zhang}.

\subsection{Numerical Verification of Theorem \ref{2order}}\label{confirm}
{\bf Problem 4.1.} Consider
\begin{equation}\label{prob3.1}
-(\beta u')' = f(x),  \quad u(0)=u(1)=0,
\end{equation}
where
\begin{equation}\label{fx}
f(x)=
\begin{cases}
x^m & x\in [0,\alpha), \\
(x-1)^m & x\in (\alpha,1].
\end{cases}
\end{equation}
 $m$ is a nonnegative integer. The interface point is located at $\alpha$ and

\begin{equation}\label{move}
\beta(x) =  \begin{cases}
      \beta^{-} & x\in [0,\alpha), \\
	 \beta^{+} & x\in (\alpha,1].
   \end{cases}
\end{equation}
The interface jump conditions are
\[   [u]_\alpha=\gamma[u']_\alpha\quad \text{ and }\quad  [\beta u']_\alpha=0.\]
where $\gamma = -\lambda\beta^{+}\beta^{-}/[\beta]_{\alpha}$. In the numerical experiment,
we set $\lambda=1$, $\beta^{-}=100$, $\beta^{+}=1$, $\alpha=1/\pi$ and $m=6$.

The exact solution is

\begin{equation}\label{exctP}
u(x) =  \begin{cases}
     \displaystyle\frac{-1}{(m+1)(m+2)\beta^{-}}x^{m+2}+c_1x & x\leq\alpha, \\
	\displaystyle \frac{-1}{(m+1)(m+2)\beta^{+}}(x-1)^{m+2}+c_2(x-1)  & x\geq\alpha,
   \end{cases}
\end{equation}
\[
\begin{bmatrix}
c_1\\c_2
\end{bmatrix}
=\dfrac{1}
{
(\alpha-1-\gamma)\beta^-+(\gamma-\alpha)\beta^+
}
\begin{bmatrix}
\alpha-\gamma-1&\beta^+\\
\alpha-\gamma &\beta^{-}
\end{bmatrix}
\begin{bmatrix}
R_1\\
R_2
\end{bmatrix},
\]
where
\[ R_1=\frac{\alpha^{m+1}}{m+1}-\frac{(\alpha-1)^{m+1}}{m+1}\]
and
\[ R_2=\frac{\gamma\alpha^{m+1}}{(m+1)\beta^-}-\frac{\gamma(\alpha-1)^{m+1}}{(m+1)\beta^+}+\frac{(\alpha-1)^{m+2}}{(m+1)(m+2)\beta^+}-\frac{\alpha^{m+2}}{(m+1)(m+2)\beta^-}.\]

In Table \ref{tab:1}, we list the nodal errors, and SCN, using linear, quadratic, and cubic elements, respectively.
In Figure \ref{fig:1}, we demonstrate optimal $L^{2}$ and broken $H^{1}$ errors using linear, quadratic, and cubic elements, respectively

%  - - - - Problem 3.1
\begin{figure}[h]
\centering
\begin{subfigure}{.33\textwidth}
  \centering
  \includegraphics[width=\linewidth]{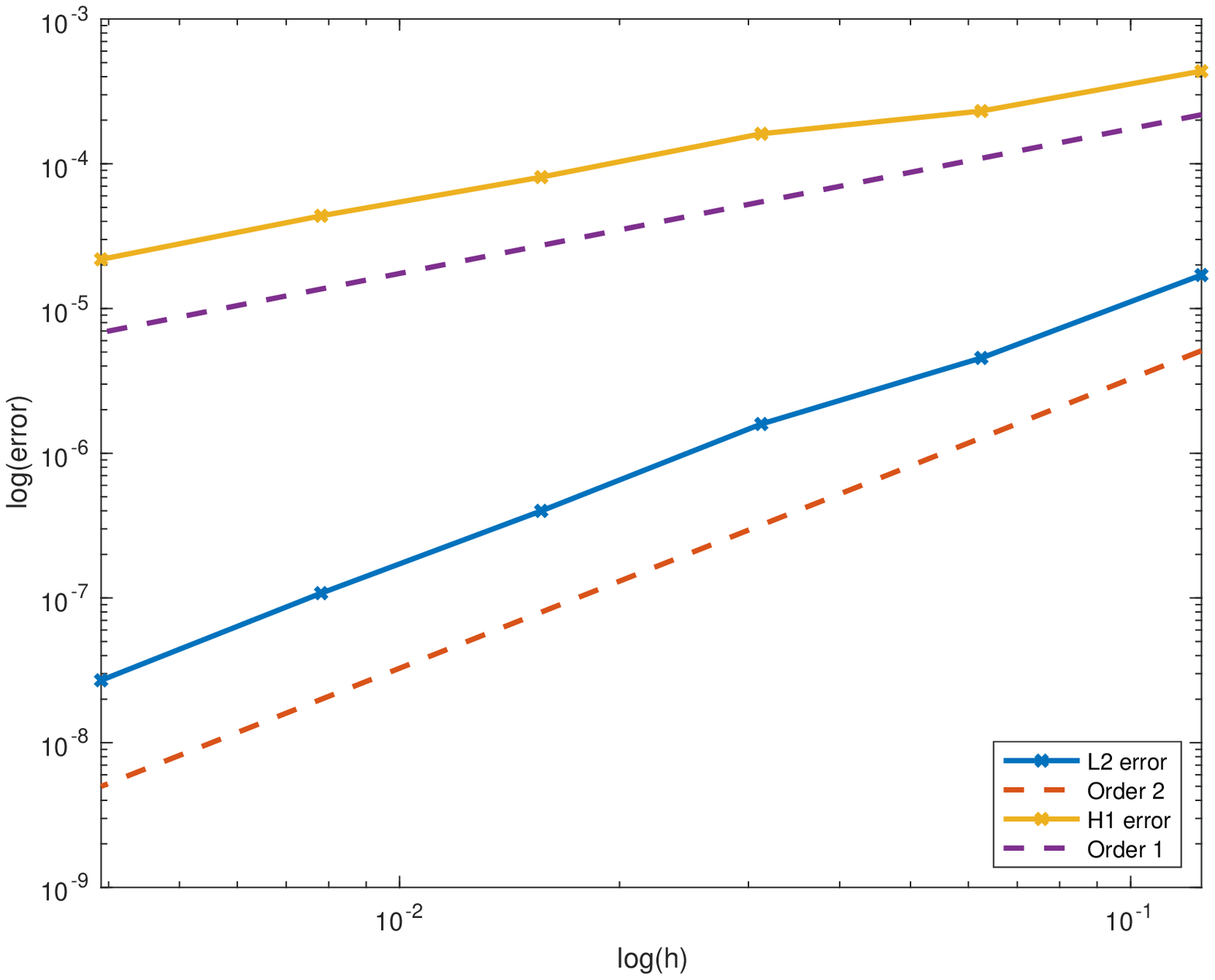}
  \caption{Linear}
  \label{fig:sub1}
\end{subfigure}%
\begin{subfigure}{.33\textwidth}
  \centering
  \includegraphics[width=\linewidth]{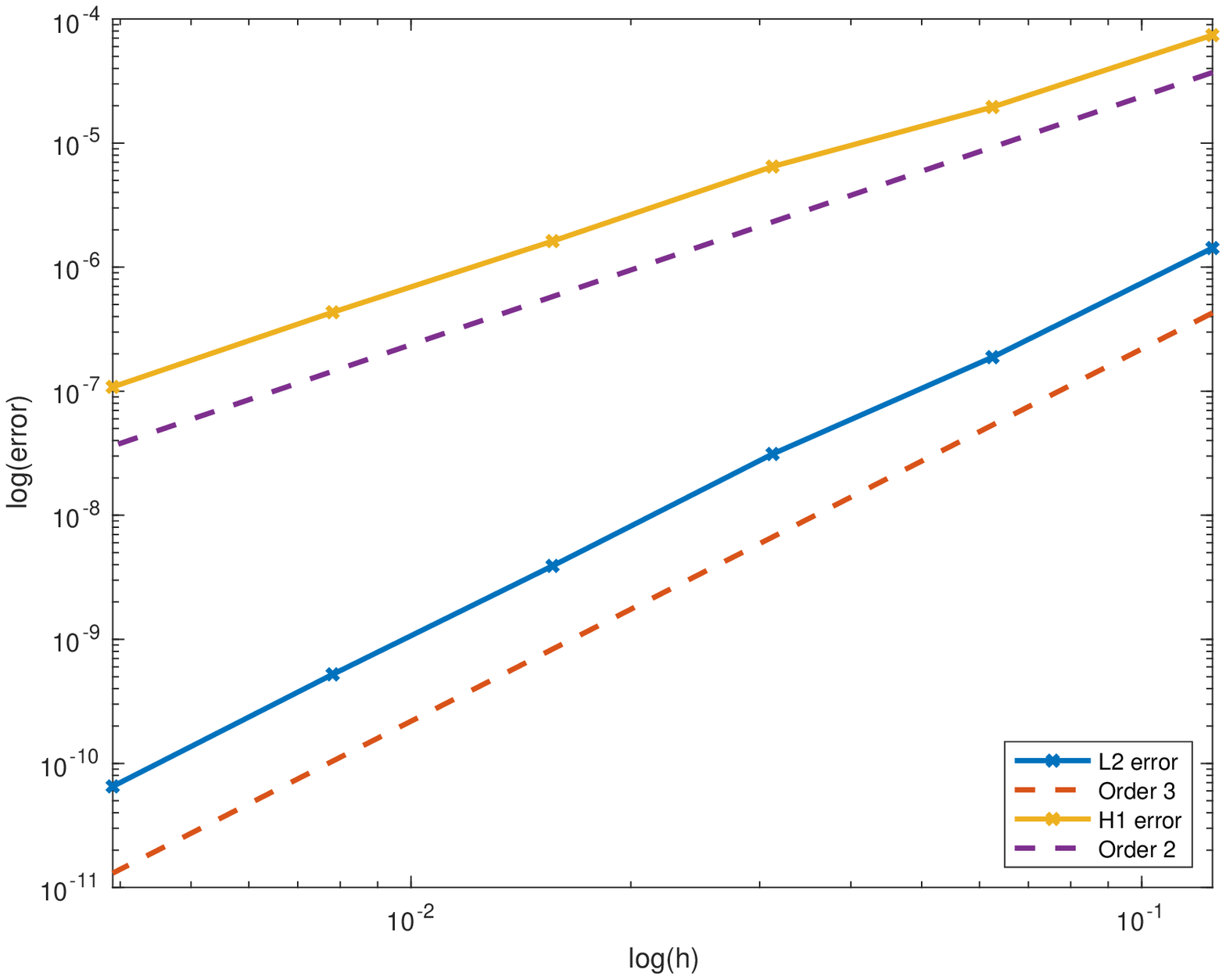}
  \caption{Quadratic}
  \label{fig:sub2}
\end{subfigure}
\begin{subfigure}{.33\textwidth}
  \centering
  \includegraphics[width=\linewidth]{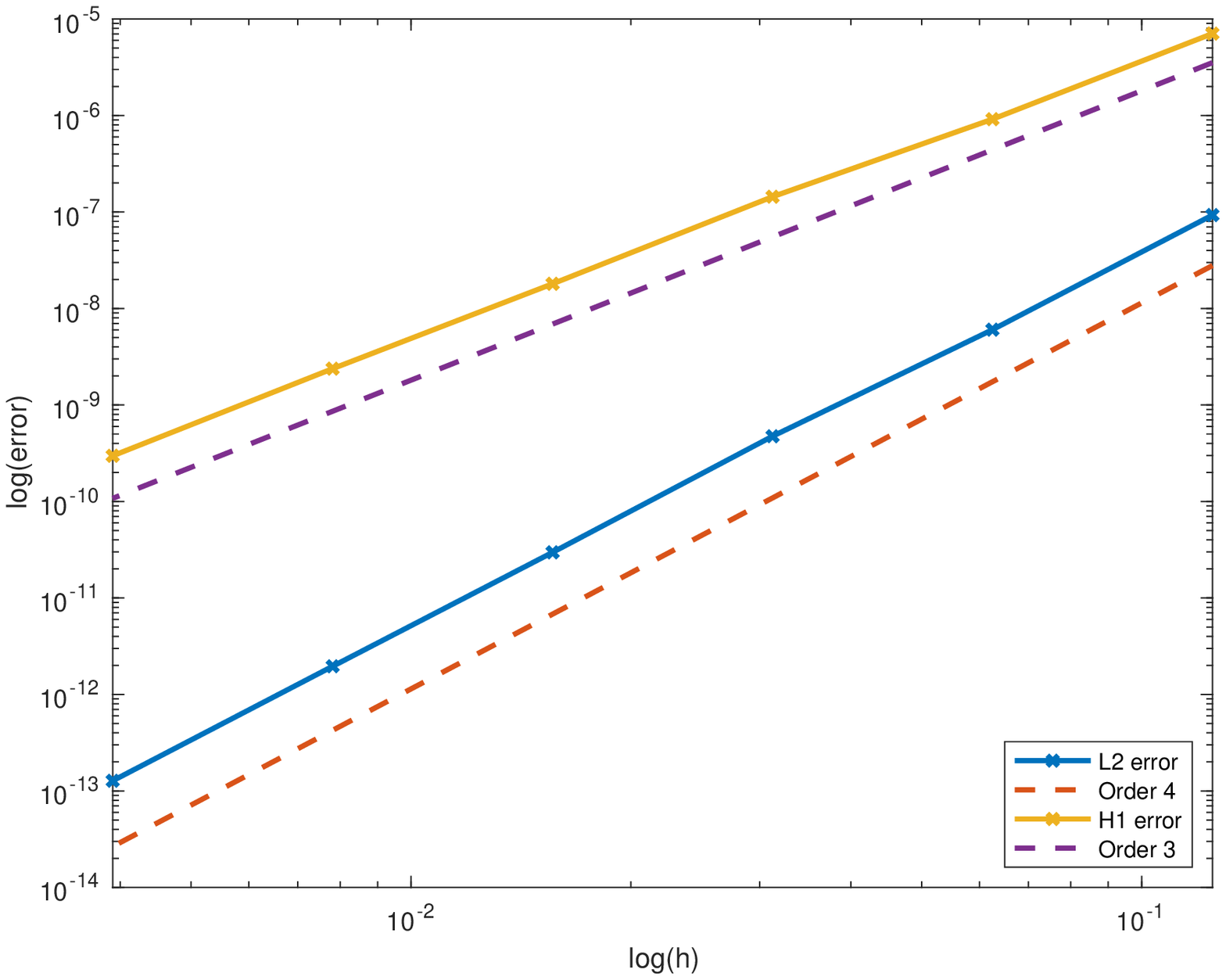}
  \caption{Cubic}
  \label{fig:sub2}
\end{subfigure}
\caption{$L_{2}$ and broken $H^{1}$ convergence rates for Problem 4.1 for linear, quadratic, and cubic basis functions where
$\beta^{-}=100$, $\beta^{+}=1$ and $m=6$.}
\label{fig:1}
\end{figure}

\begin{table}%[!h]
\begin{center}
\begin{tabular}{l cc cc cc}
\toprule
Problem 4.1 & \multicolumn{2}{c}{Linear} & \multicolumn{2}{c}{Quadratic} & \multicolumn{2}{c}{Cubic} \\
\cmidrule(lr){2-3} \cmidrule(lr){4-5} \cmidrule(lr){6-7}
 & Nodal error & SCN & Nodal error & SCN & Nodal error & SCN \\
$ h=1/8 $ &		 	 3.30e-17 &	 	 2.15e3  &		 1.06e-16 &	 	 1.07e4 &		 4.86e-17 &	 	 4.07e4 \\
 \hline
$ h=1/16 $ &			 2.46e-16 &	 	 4.12e3 &	  8.67e-17 &		 1.54e4 &	 	 2.47e-16 &		 1.21e5 \\
 \hline
$ h=1/32 $ &		 	 5.50e-16 &	 	 8.58e3 &	 	 1.10e-15 &		 5.38e4 &		 1.65e-16 &	 	 2.51e5 \\
 \hline
$ h=1/64 $ &	 	 	 1.45e-15 &	 	 1.65e4 &		 2.05e-16 &	 	 6.77e4 &	  2.39e-15 &	 	 8.64e4 \\
 \hline
$ h=1/128 $ &	 	 6.96e-15 &		 3.43e4 &	 	 4.99e-16 &		 7.68e4 &	 	 7.30e-15 &		 1.73e5 \\
 \hline
$ h=1/256 $ &		 	 1.66e-15 &	 	 6.69e4 &	 	 3.94e-15 &	 	 4.75e5  &	 	 6.66e-14 &	 	 3.46e5 \\
\bottomrule
\end{tabular}
 \caption{$L^{2}$, broken $H^{1}$, nodal errors, and scaled condition numbers (SCN) with discontinuous jump conditions for Problem 4.1  for linear, quadratic, and cubic basis functions where $\beta^{-}=100$ = $\beta^{+}=1$ and $m=6$.}
 \label{tab:1}
\end{center}
\end{table}

 % - - - - - P 3.1

{\bf Problem 4.2.} Consider a 2D interface problem on the unit square
\begin{equation}\label{prob3.2}
-\nabla \cdot (\beta \nabla U) = F(x,y), \text{ on } \Omega=[0,1]^2,\quad    U=0 \text{ on }\partial \Omega,
\end{equation}
with a vertical interface line $x=\alpha$ dividing $\Omega$ and the material property $\beta$ depends only on $x$:
\begin{equation}\label{move}
\beta(x,y)=\beta(x) =  \begin{cases}
      \beta^{-} & x\in [0,\alpha), \\
	 \beta^{+} & x\in (\alpha,1].
   \end{cases}
\end{equation}
The interface jump conditions are
\[   [U]_\alpha=\gamma[\nabla U\cdot {\bf n}]_\alpha\quad \text{ and }\quad  [\beta \nabla U\cdot {\bf n}]_\alpha=0\]
where $\gamma = -\lambda\beta^{+}\beta^{+}/[\beta]_{\alpha}$.

For this example we take $U(x,y)=u(x)Y(y)$ with $Y(y)=y(y-1)$ and
\begin{equation}\label{n3}  u(x)=
\begin{cases}
\beta^{+}\sin \pi x   \quad x\leq \alpha\\
\beta^{-}\sin \pi x      \quad x>\alpha
\end{cases}
\end{equation}
 so that the boundary conditions and the jump conditions
can be easily satisfied. The right hand side function $F(x,y)$ in \eqref{prob3.2} is then defined by
  \begin{equation}\label{n4}
F(x,y)=\beta^+\beta^-\sin\pi x(\pi^2y(y-1)-2).
 \end{equation}
Notice that $\gamma=-\frac 1 \pi\tan(\pi\alpha)$.
In Figure \ref{fig:2}, we demonstrate optimal $L^{2}$ and broken $H^{1}$ errors using linear, quadratic, and cubic elements, respectively.
In Table \ref{tab:2}, we list the nodal errors, and SCN, using linear, quadratic, and cubic elements, respectively.

%  - - - - Problem 4
\begin{figure}[h]
\centering
\begin{subfigure}{.33\textwidth}
  \centering
  \includegraphics[width=\linewidth]{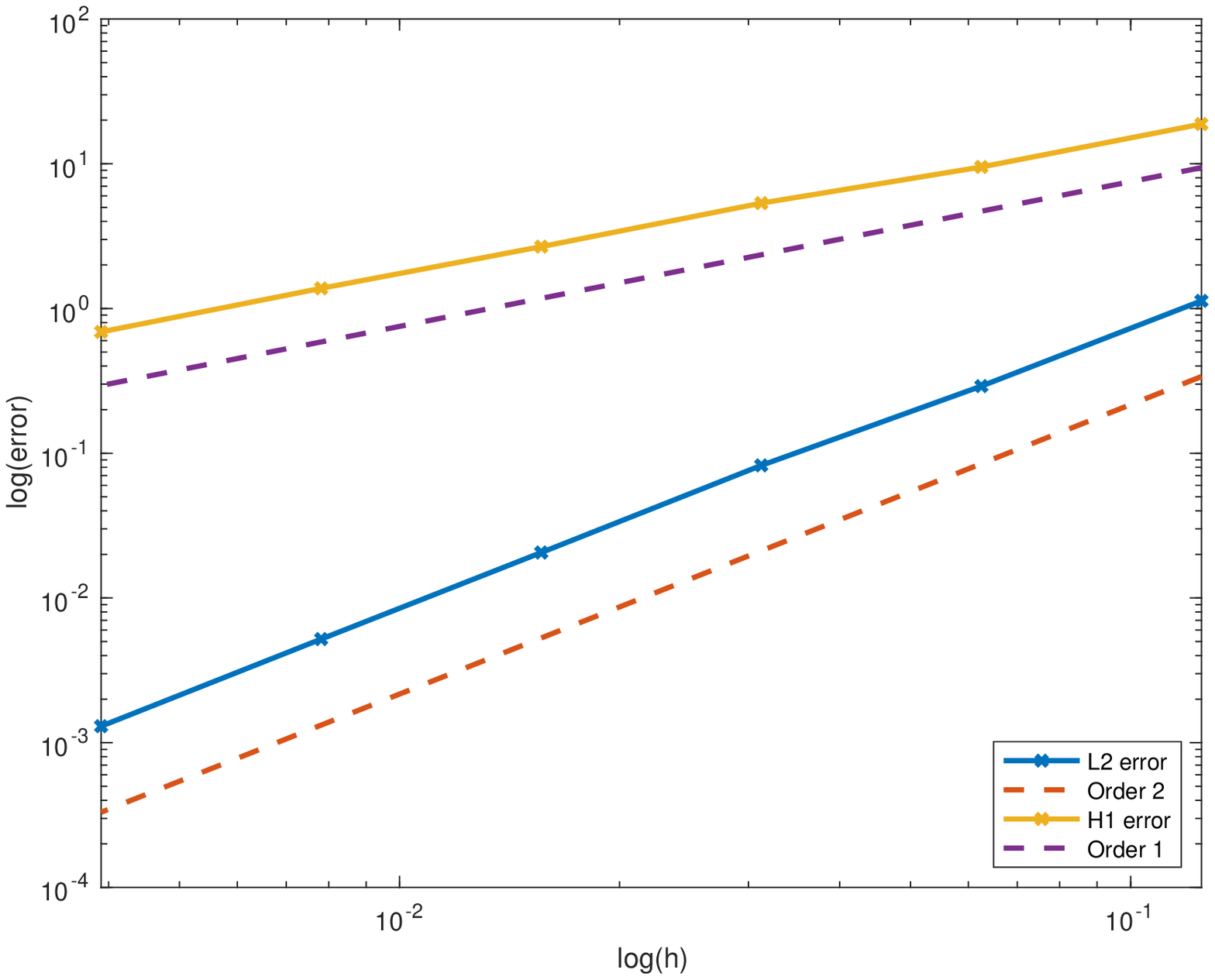}
  \caption{Linear}
  \label{fig:sub1}
\end{subfigure}%
\begin{subfigure}{.33\textwidth}
  \centering
  \includegraphics[width=\linewidth]{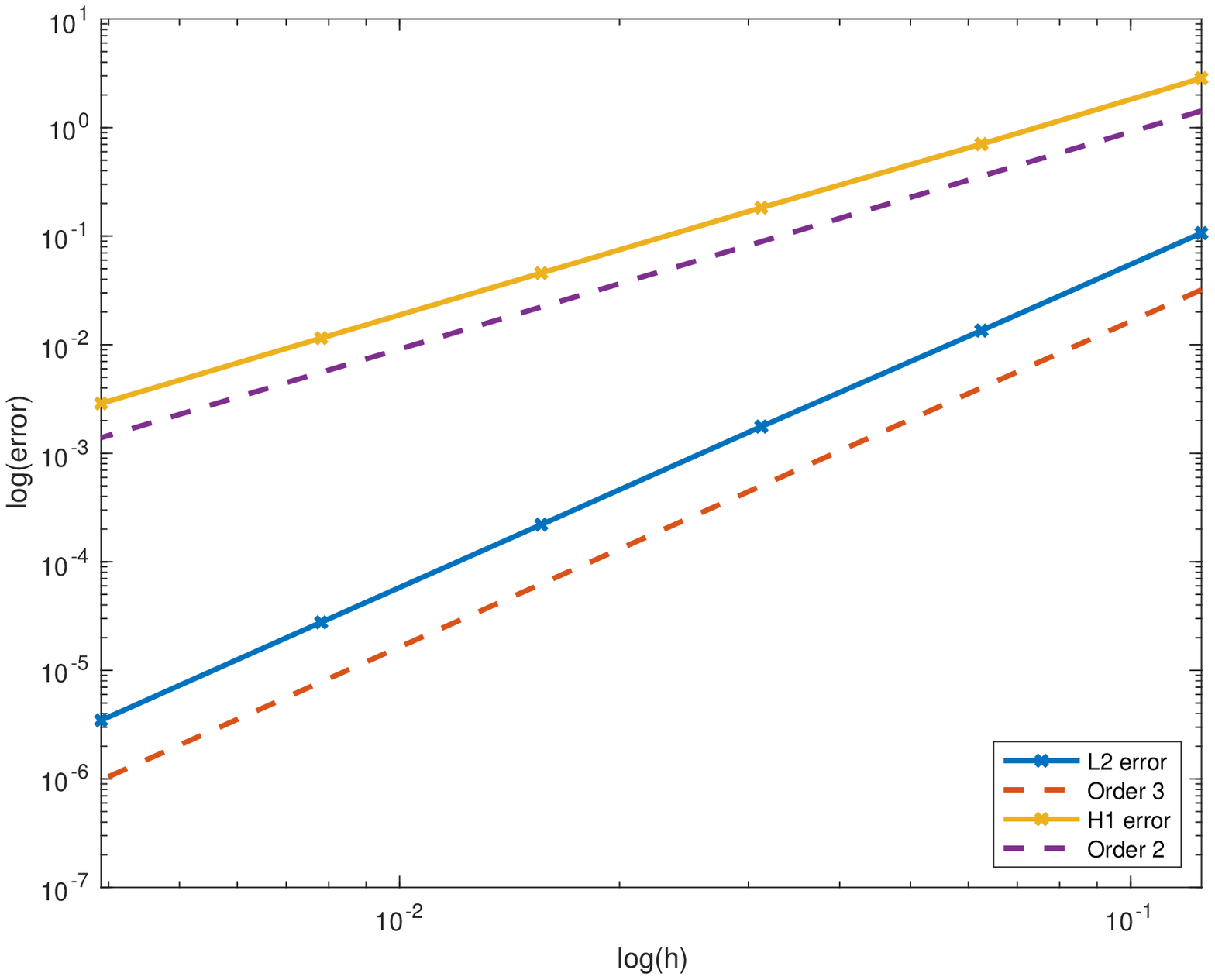}
  \caption{Quadratic}
  \label{fig:sub2}
\end{subfigure}
\begin{subfigure}{.33\textwidth}
  \centering
  \includegraphics[width=\linewidth]{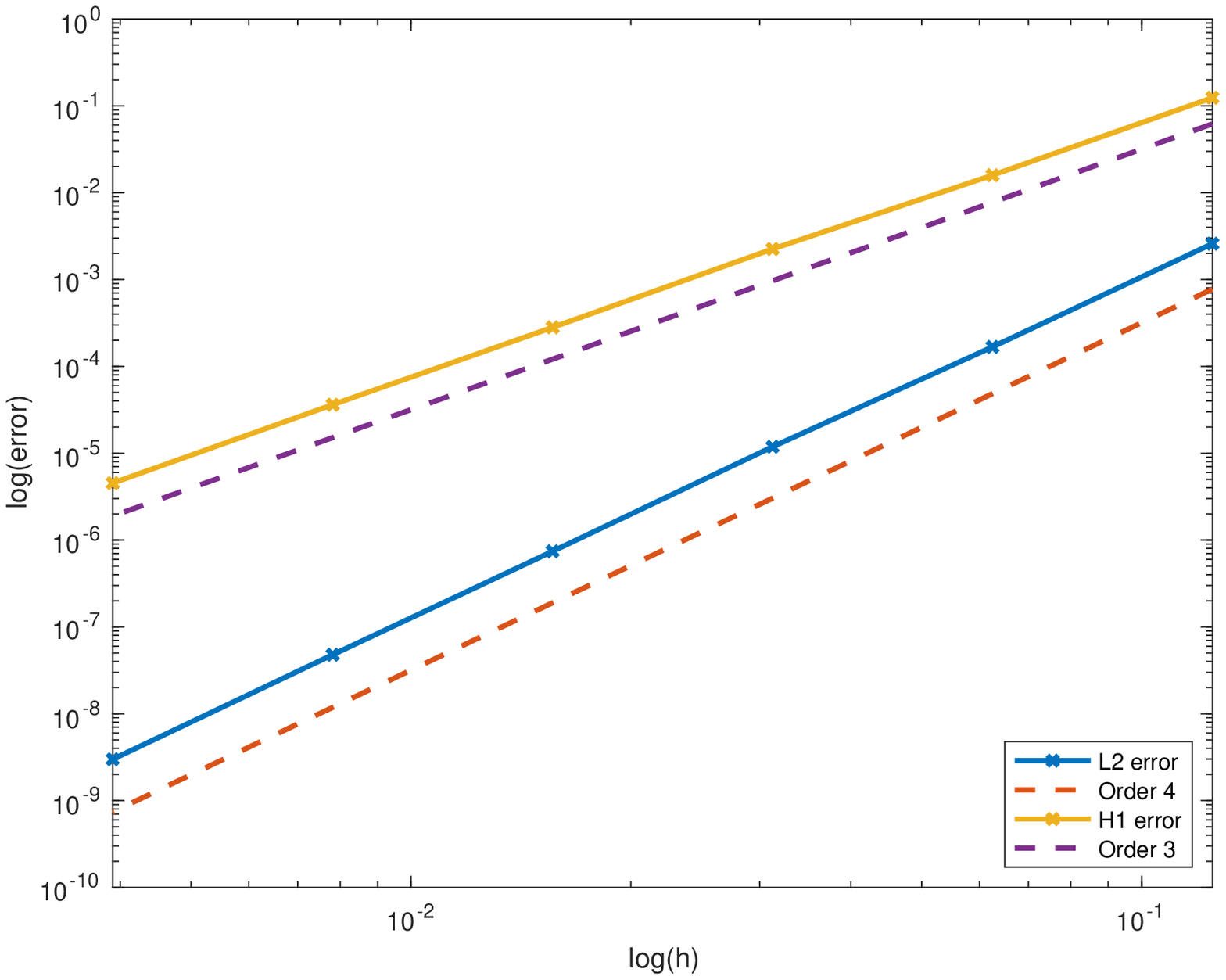}
  \caption{Cubic}
  \label{fig:sub2}
\end{subfigure}
\caption{$L_{2}$ and broken $H^{1}$ convergence rates for Problem 4.2 for linear, quadratic, and cubic basis functions}
\label{fig:2}
\end{figure}

\begin{table}%[!h]
\begin{center}
\begin{tabular}{l cc cc cc}
\toprule
Problem 4.2 & \multicolumn{2}{c}{Linear} & \multicolumn{2}{c}{Quadratic} & \multicolumn{2}{c}{Cubic} \\
\cmidrule(lr){2-3} \cmidrule(lr){4-5} \cmidrule(lr){6-7}
 & Nodal error & SCN & Nodal error & SCN & Nodal error & SCN \\
$ h=1/8 $ &	6.54e-1 & 9.04e2  & 3.79e-3 & 5.98e5 & 5.78e-4 & 2.26e5 \\
 \hline
$ h=1/16 $ &	 2.44e-1 & 2.71e4 & 3.04e-4 & 6.32e6 & 4.82e-5 & 8.10e5 \\
 \hline
$ h=1/32 $ &	 6.10e-2 & 4.55e4 & 1.90e-5 & 7.98e5 & 3.10e-6 & 3.46e7\\
 \hline
$ h=1/64 $ &		1.68e-2 & 1.92e5 & 1.26e-6 & 4.82e6 & 2.02e-7 & 2.71e8 \\
 \hline
$ h=1/128 $ &	4.21e-3 & 4.77e5	 & 7.87e-8 & 1.32e7 & 1.24e-8 & 1.68e9\\
 %\hline
%$ h=1/256 $ &	
 %
 %
\bottomrule
\end{tabular}
 \caption{$L^{2}$, broken $H^{1}$, nodal errors, and scaled condition numbers (SCN) with discontinuous jump conditions for  Problem 4.2   for linear, quadratic, and cubic basis functions.}
 \label{tab:2}
\end{center}
\end{table}

\newpage

\subsection{Multi-layer Porous Wall Model}\label{quacubic}
% - - - - - - - - Problem # 2
In this subsection, we test our method using the multi-layer porous wall model for the drug-eluting stents \cite{Pontrelli}. In this one-dimensional wall model of layers, a drug is injected or released at an interface and gradually diffuses rightward. The concentration is thus discontinuous across the injection interface and continuous in the other layers. At all interface points, a zero-flux condition is imposed. We run tests using the enriched linear, quadratic, and cubic finite element spaces.
For Problem 1, we place only one interface point to model the layer where the drug is delivered.
For Problem 2, we place two interfaces to model the layers where the concentration is continuously spread.
Finally, for Problem 3 we combine the previous two cases and place three interface points to simulate the full wall model. In each of the three problems, we display run results using linear, quadratic, and cubic-enriched finite elements. We confirm that our method is indeed an SGFEM or SSGFEM \cite{Babuska1, Babuska2, Deng, zhang2019strongly,zhang2020strongly} in the sense that the condition numbers are comparable with those in the continuous case with robustness, and that it has optimal order convergence
  in the $L^2$ and broken $H^1$-norms.

{\bf Problem 1.  Discontinuous Solution.} Consider the two-point boundary value problem with one interface point $\alpha_0=1/9$
\begin{equation}\label{Pro1}
\frac{\partial}{\partial x}\left(-D\frac{\partial u}{\partial x} + 2\delta u \right) + \eta u = f \quad \text{ in }(0,1)
\end{equation}
subject to the the no-flux Neumann condition at $x=0$ and the Dirichlet condition at $x=1$:
\[
D_{0}u'(0) = 0, \quad u(1) = \frac{1}{3}.
\]
Here the drug reaction coefficient $\eta=0$, and the drug diffusivity $D$ and the characteristic convection parameter $\delta$ are piecewise continuous with respect to  $[0,1/9]$ and $[1/9,1]$:
\[D(x) =
\begin{cases}
D_{0} = 1 & x\in[0,1/9] \\
D_1 = \frac{18(n-1)}{10n} & x\in[1/9,1];
\end{cases}\\
\]
\[
\delta(x)=
\begin{cases}
\delta_{0} = 0 & x\in[0,1/9] \\
 \delta = 0.5(9nD_1-8.1(n-1))& x\in[1/9,1].
 \end{cases}
\]
Furthermore, at the interface point $\alpha_{0}$, one of the jump conditions is implicit
\begin{equation}\label{disContJump}
\begin{cases}
[u]_{\alpha_{0}} = \lambda D_{0}u'(\alpha_{0}), \\
-D_{0}u'(\alpha_{0})= -D_{1}u'(\alpha_{0}^{+}) + 2\delta_{1}u(\alpha_{0}^{+})
\end{cases}
\end{equation}
where $\lambda = 1/81(n-1)D_{0}$.
The exact solution
\[u(x) =
\begin{cases}
u_{0} = x^{n-1}/30, &  x\in[0,1/9],\\
 u_1 = x^n/3,& x\in[1/9,1]. \\
\end{cases}
\]
We test the effectiveness of the method with $n=4$ and with the enrichment function in \eqref{eqn:lin1}. The test results using linear, quadratic, and cubic elements are displayed in Table \ref{tab:3} and Figure \ref{fig:3}, respectively.

%  - - - - Problem 1
\begin{figure}[h]
\centering
\begin{subfigure}{.33\textwidth}
  \centering
  \includegraphics[width=\linewidth]{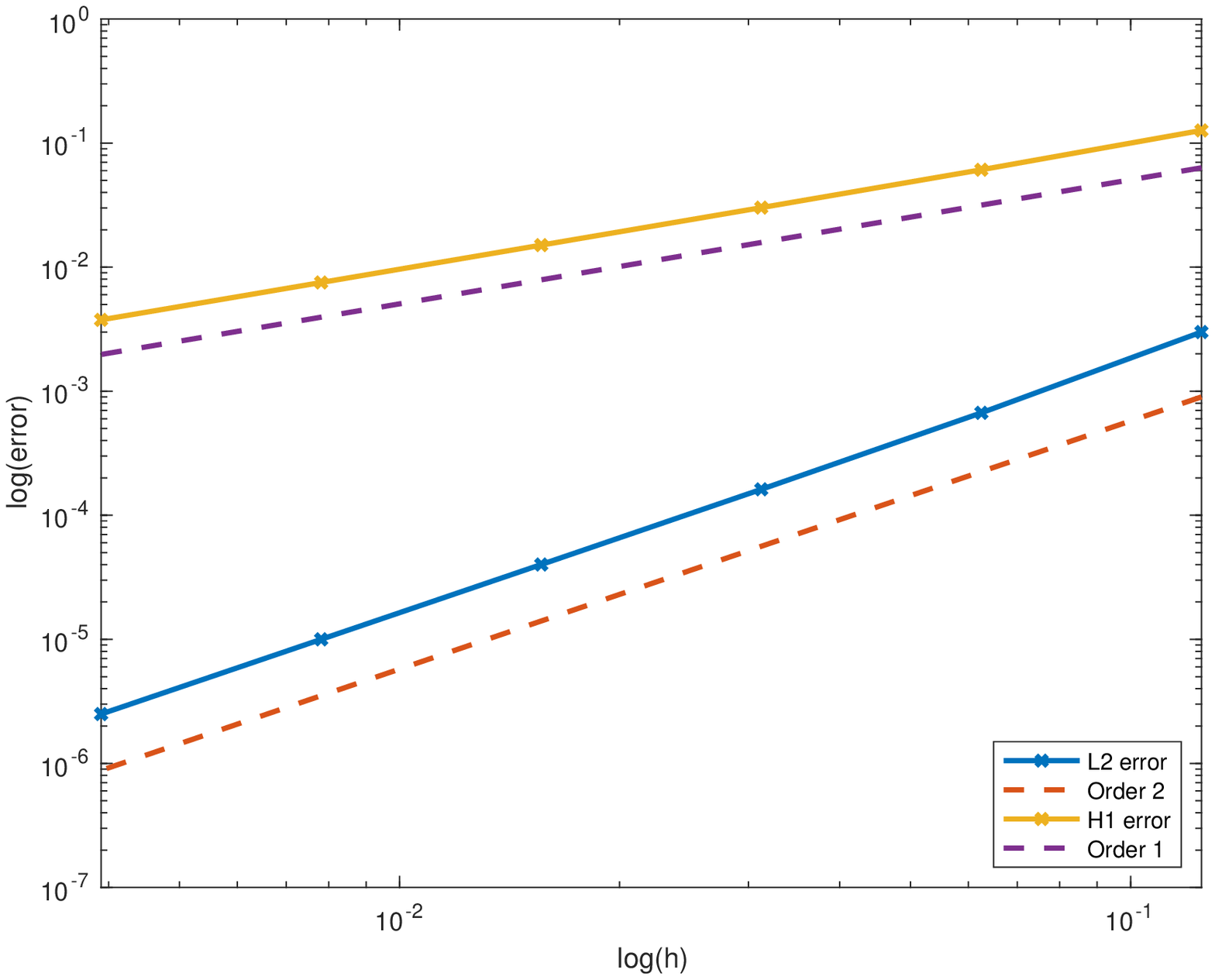}
  \caption{Linear}
  \label{fig:sub1}
\end{subfigure}%
\begin{subfigure}{.33\textwidth}
  \centering
  \includegraphics[width=\linewidth]{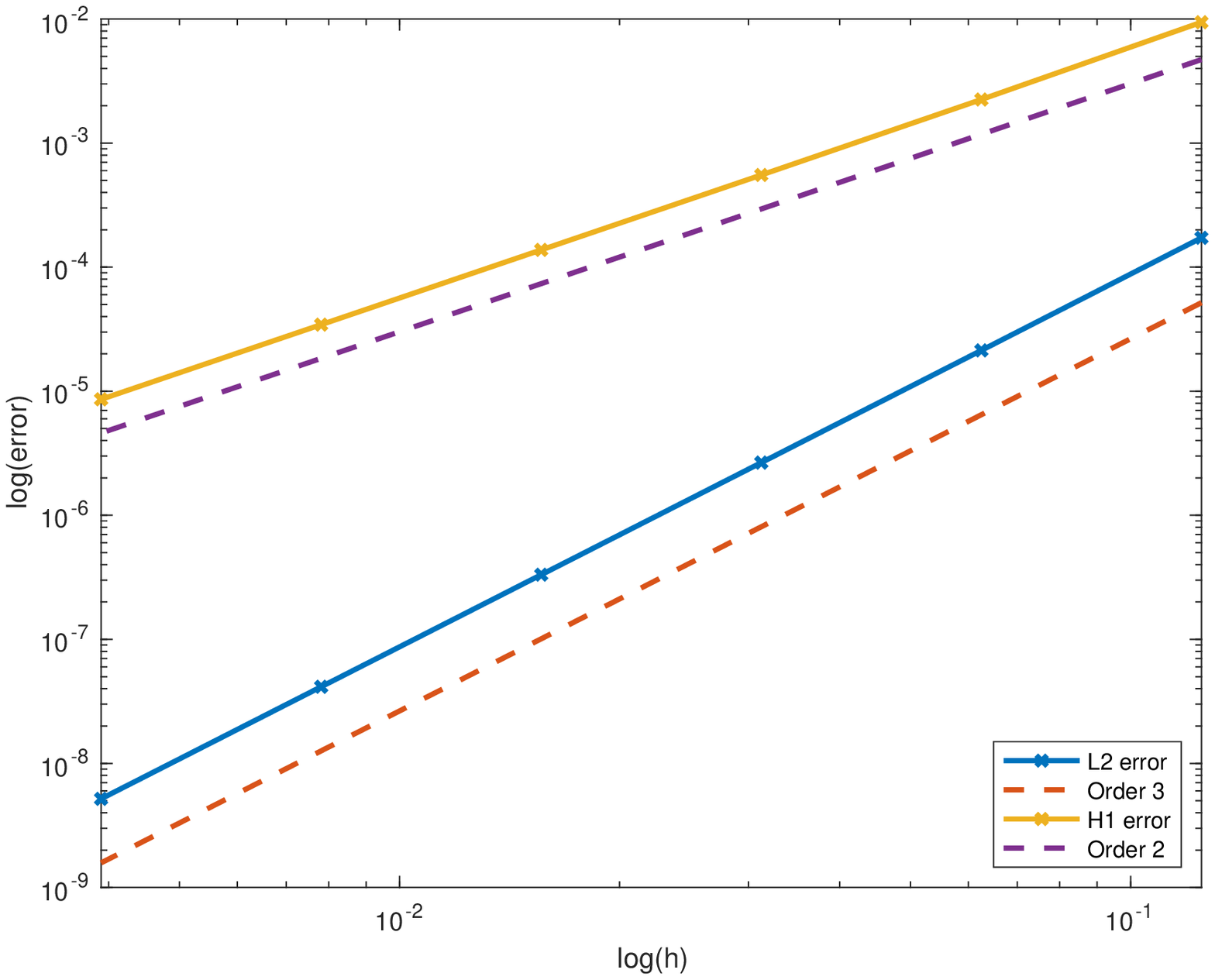}
  \caption{Quadratic}
  \label{fig:sub2}
\end{subfigure}
\begin{subfigure}{.33\textwidth}
  \centering
  \includegraphics[width=\linewidth]{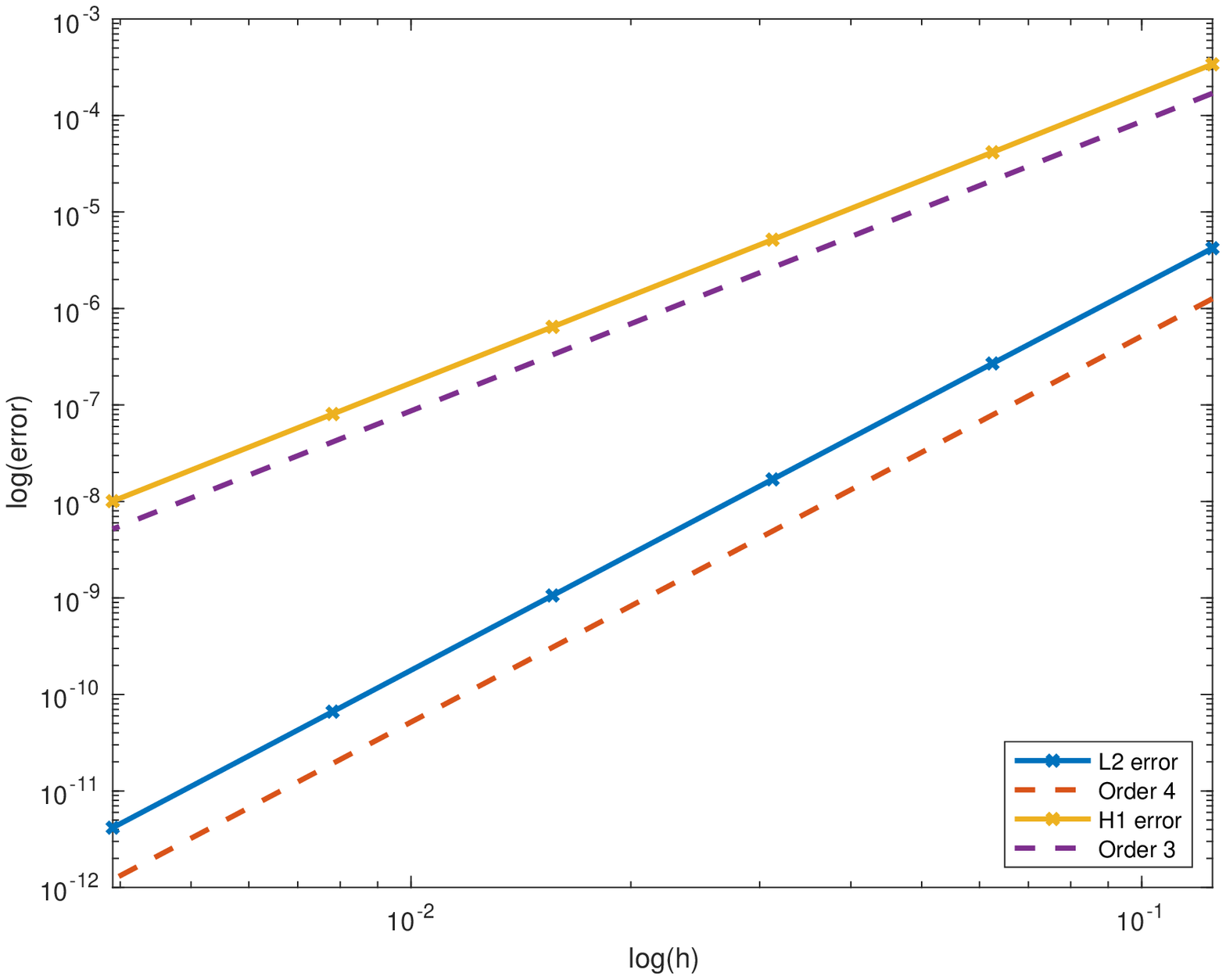}
  \caption{Cubic}
  \label{fig:sub2}
\end{subfigure}
\caption{$L_{2}$ and broken $H^{1}$ convergence rates for Problem 1 for linear, quadratic, and cubic basis functions}
\label{fig:3}
\end{figure}

\begin{table}%[!h]
\begin{center}
\begin{tabular}{l cc cc cc}
\toprule
Problem 1 & \multicolumn{2}{c}{Linear} & \multicolumn{2}{c}{Quadratic} & \multicolumn{2}{c}{Cubic} \\
\cmidrule(lr){2-3} \cmidrule(lr){4-5} \cmidrule(lr){6-7}
 & Nodal error & SCN & Nodal error & SCN & Nodal error & SCN \\
 $ h=1/8 $ &	 		 1.27e-2 &		 9.94e3 &	 	 2.81e-4 &	 	 7.57e8 &		 2.37e-6 &		 5.56e8 \\
 \hline
$ h=1/16 $ &		 	 2.63e-3 &	 	 4.70e8 &	 	 1.70e-5 &	 	 1.22e7 &		 4.00e-8 &	 	 1.35e7 \\
 \hline
$ h=1/32 $ &	 	 5.86e-4 &		 3.89e6 &	 	 1.09e-6 &		 4.06e6 &		 6.46e-10 &		 3.33e6 \\
 \hline
$ h=1/64 $ &	 	 1.46e-4 &		 1.03e7 &	 	 6.75e-8 &	 	 3.58e6  &	 	 1.02e-11 &	 	 5.19e6 \\
 \hline
$ h=1/128 $ &		 	 3.63e-5 &	 6.13e6 &		 4.22e-9 &	  1.29e6 &		 6.94e-14 &		 1.56e7 \\
 \hline
$ h=1/256 $ &		 	 9.06e-6 &	 1.34e7 &		 2.64e-10 &	 	 2.14e6 &		 8.35e-13 &		 5.14e7 \\
\bottomrule
\end{tabular}
 \caption{$L^{2}$, broken $H^{1}$, nodal errors, and scaled condition numbers (SCN) with discontinuous jump conditions for Problem 1  for linear, quadratic, and cubic basis functions.}
 \label{tab:3}
\end{center}
\end{table}

{\bf Problem 2. Continuous Solution.} Consider the two-point boundary value problem with two interface points $\alpha_1=1/3,\alpha_2= 2/3$
\begin{equation}\label{P2}
\frac{\partial}{\partial x}\left(-D\frac{\partial u}{\partial x} + 2\delta u \right) + \eta u = f, \quad x\in(0,1)
\end{equation}
with the boundary conditions
\[
D_{0}u'(0) = 0 \quad u(1) = 0.
\]
Here with $n=4$
\[D(x) =
\begin{cases}
D_1 = \frac{18(n-1)}{10n} & x\in[0,1/3] \\
D_2 = \frac{6nD_1 - 2\delta}{3(n+1)}  & x\in[1/3,2/3] \\
D_3 = \frac{8\delta_2 - 3(n+1)D_2}{3(n+5)}  & x\in[2/3,1];
\end{cases}
\]

\[\delta(x)=
\begin{cases}
 \delta = 0.5(9nD_1-8.1(n-1)) & x\in[0,1/3]\\
 \delta_2 = 0.5(3(n+1)D_2 - 3nD_1 + 2\delta)  & x\in[1/3,2/3] \\
\delta_3 = 0.25(3(n-1)D_3 - 3(n+1)D_2 + 4\delta_2) & x\in[2/3,1],
\end{cases}
\]
 and $\eta = 10, 1, 0.1$ in respective subintervals.
 At the interface points $\alpha_{i}$ for $i=1,2$, the solution $u$ is continuous and
\begin{equation}
\begin{cases}\label{contJump}
[u]_{\alpha_{i}} = 0, \\
-D_{i}u'(\alpha_{i}^{-}) + 2\delta_{i}u(\alpha_{i}^{-}) = -D_{i+1}u'(\alpha_{i}^{+}) + 2\delta_{i+1}u(\alpha_{i}^{+}).
\end{cases}
\end{equation}
The exact solution is

\[u(x) =
\begin{cases}
 u_1 = x^n/3& x\in[0,1/3] \\
u_2 = x^{n+1} & x\in[1/3,2/3] \\
u_3 = 3(1-x)x^{n+1}& x\in[2/3,1].
\end{cases}
\]

The enrichment function $\psi$ is well-known \cite{Babuska1, Deng, ATTChou2021}:
\begin{equation}
\psi(x) =
\begin{cases}
0   &  x\in [0,x_{k}]\\
\dfrac{(x_{k+1}-\alpha)(x_{k}-x)}{x_{k+1}-x_{k}} &  x\in [x_{k},\alpha]  \label{eqn:lin}\\
\dfrac{(\alpha-x_{k})(x-x_{k+1})}{x_{k+1}-x_{k}} &  x\in [\alpha,x_{k+1}]\\
0   &  x\in [x_{k+1},1].
\end{cases}
\end{equation}
 The test results using linear, quadratic, and cubic elements are displayed in Table \ref{tab:4} and Figure \ref{fig:4}, respectively.

%  - - - - Problem 2
\begin{figure}[h]
\centering
\begin{subfigure}{.33\textwidth}
  \centering
  \includegraphics[width=\linewidth]{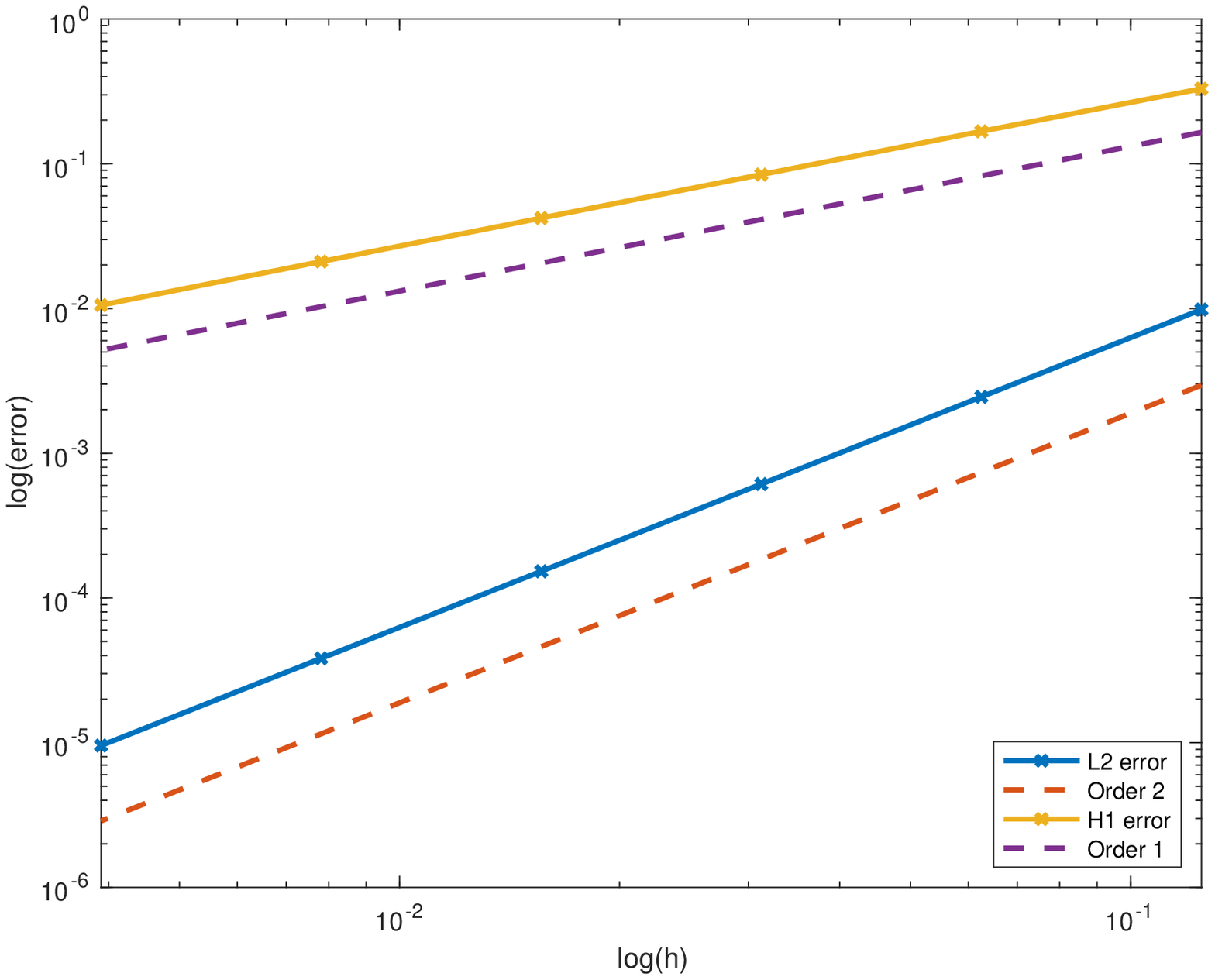}
  \caption{Linear}
  \label{fig:sub1}
\end{subfigure}%
\begin{subfigure}{.33\textwidth}
  \centering
  \includegraphics[width=\linewidth]{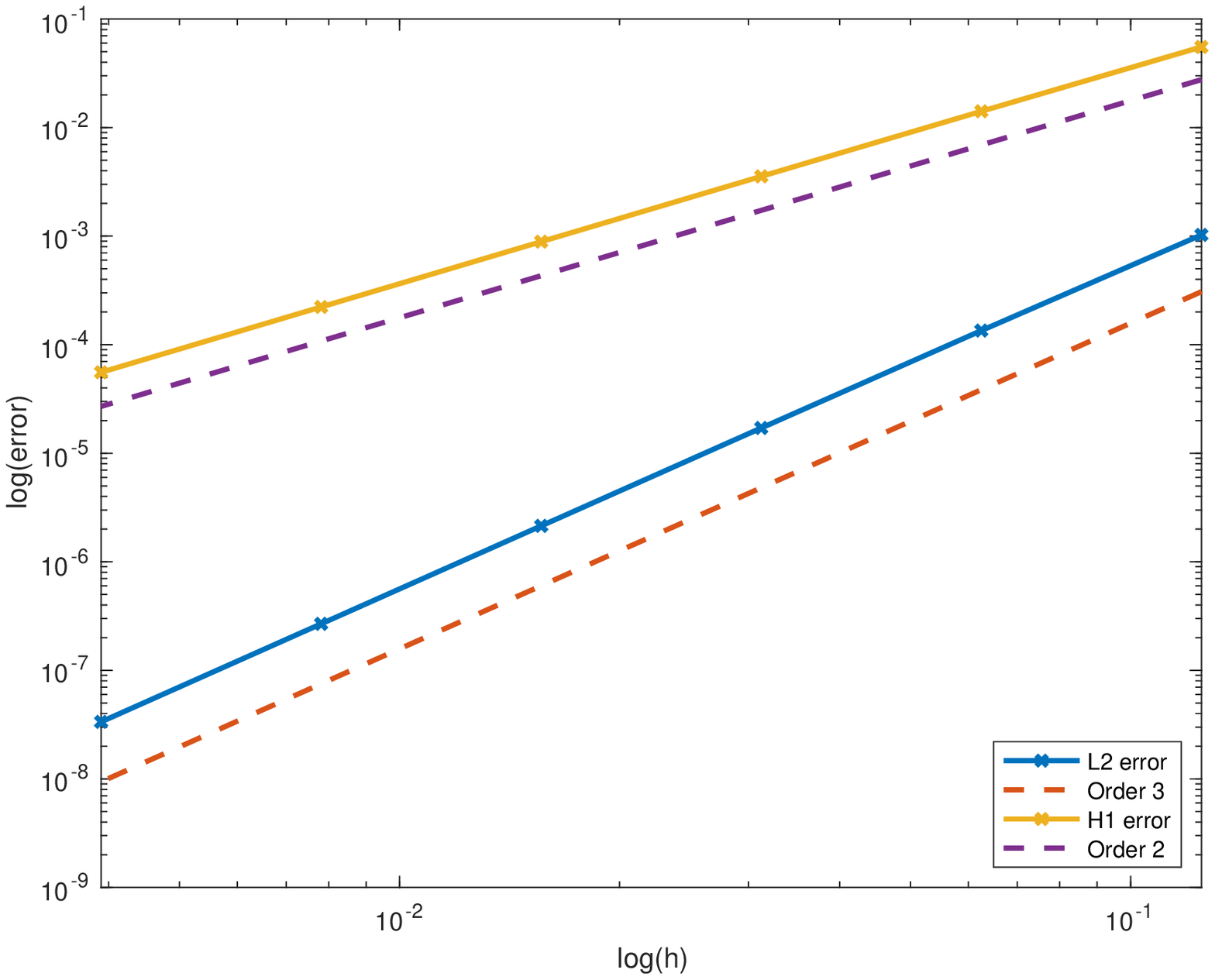}
  \caption{Quadratic}
  \label{fig:sub2}
\end{subfigure}
\begin{subfigure}{.33\textwidth}
  \centering
  \includegraphics[width=\linewidth]{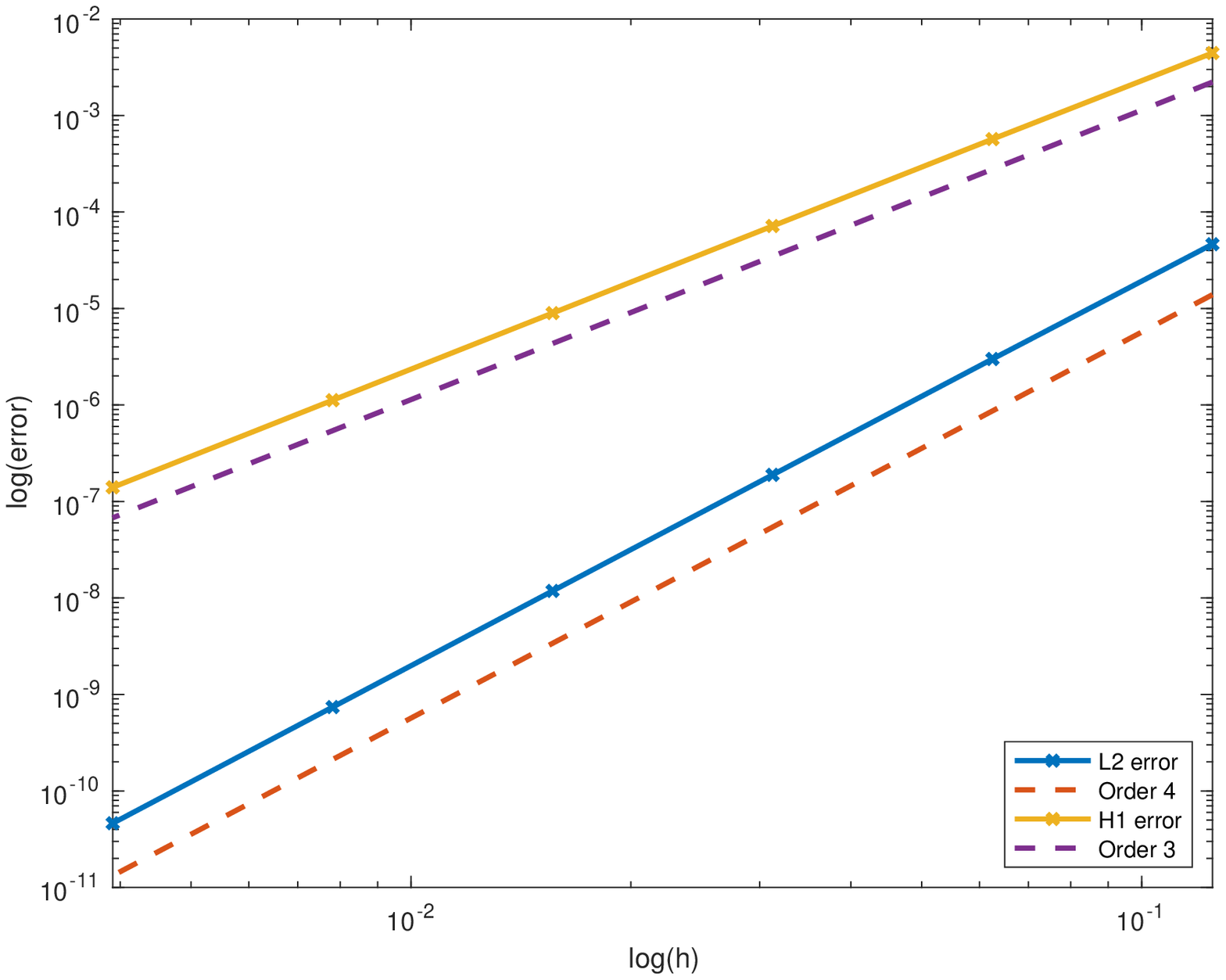}
  \caption{Cubic}
  \label{fig:sub2}
\end{subfigure}
\caption{$L_{2}$ and broken $H^{1}$ convergence rates for Problem 2 for linear, quadratic, and cubic basis functions}
\label{fig:4}
\end{figure}

\begin{table}%[!h]
\begin{center}
\begin{tabular}{l cc cc cc}
\toprule
Problem 2 & \multicolumn{2}{c}{Linear} & \multicolumn{2}{c}{Quadratic} & \multicolumn{2}{c}{Cubic} \\
\cmidrule(lr){2-3} \cmidrule(lr){4-5} \cmidrule(lr){6-7}
 & Nodal error & SCN & Nodal error & SCN & Nodal error & SCN \\
 $ h=1/8 $ &	 	 2.92e-2 &		 1.11e3 &	 	 1.02e-3 &	 	 9.66e8 &	 	 1.16e-5 &		 9.80e6 \\
 \hline
$ h=1/16 $ &		 	 6.43e-3 &		 1.74e7 &		 5.81e-5 &	 	 1.82e7 &	 	 1.77e-7 &		 7.07e6 \\
 \hline
$ h=1/32 $ &	 	 1.53e-3 &		 4.78e7 &		 3.72e-6 &	 	 1.16e7 &	  2.85e-9 &		 1.79e7 \\
 \hline
$ h=1/64 $ &		 3.82e-4 &	 	 1.86e7 &		 2.31e-7 &	 	 8.15e6 &	 	 4.49e-11 &		 3.73e7 \\
 \hline
$ h=1/128 $ &	 	 9.51e-5 &	 	 3.94e7 &		 1.44e-8 &		 7.05e6 &	 9.76e-13 &	 	 2.90e9 \\
 \hline
$ h=1/256 $ &		 	 2.38e-5 &	 	 5.57e7 &		 9.01e-10 &		 7.14e6 &	 6.82e-13 &	 1.02e8 \\
\bottomrule
\end{tabular}
 \caption{$L^{2}$, broken $H^{1}$, nodal errors, and scaled condition numbers (SCN) with discontinuous jump conditions for Problem 2  for linear, quadratic, and cubic basis functions.}
 \label{tab:4}
\end{center}
\end{table}

%%%%%%%%%%%%%%%%%%%%%%%%%%%%%%%%%%%%%%%%%%%%%%%%%%%%%%%%%%%%%%%

{\bf Problem 3. Implicit and Explicit Conditions Both Present.} In this problem, we combine the interfaces of the last two problems. The interface points are $\alpha_{0} = 1/9$,
$\alpha_{1} = 1/3$ and $\alpha_{2} = 2/3$.
The two-point boundary value problem is
\begin{equation}\label{P3}
\frac{\partial}{\partial x}\left(-D\frac{\partial u}{\partial x} + 2\delta u \right) + \eta u = f \quad x\in(0,1)
\end{equation}
subject to the boundary conditions
\[
D_{0}u'(0) = 0, \qquad u(1) = 0.
\]
The coefficients are defined as follows:

\[
D(x) =
\begin{cases}
D_{0} = 1 & x\in [0,1/9] \\
D_1 = \frac{18(n-1)}{10n}, & x\in [1/9,1/3] \\
D_2 = \frac{6nD_1 - 2\delta}{3(n+1)}  & x\in [1/3,2/3] \\
D_3 = \frac{8\delta_2 - 3(n+1)D_2}{3(n+5)}  & x\in [2/3,1];
\end{cases}
\]

\[\delta(x)=
\begin{cases}
\delta_{0} = 0 & x\in[0,1/9] \\
 \delta = 0.5(9nD_1-8.1(n-1)) & x\in[1/9,1/3] \\
 \delta_2 = 0.5(3(n+1)D_2 - 3nD_1 + 2\delta)  & x\in[1/3,2/3] \\
\delta_3 = 0.25(3(n-1)D_3 - 3(n+1)D_2 + 4\delta_2) & x\in [2/3,1];
\end{cases}
\]
$n=4$ and $\eta =0, 10, 1, 0.1$ in respective subintervals.
The exact solution is

\[u(x) =
\begin{cases}
u_{0} = x^{n-1}/30 & [0,1/9]\\
 u_1 = x^n/3& [1/9,1/3] \\
u_2 = x^{n+1} & [1/3,2/3] \\
u_3 = 3(1-x)x^{n+1}& [2/3,1]
\end{cases}
\]
and it satisfies the jump condition (\ref{disContJump}) at $1/9$ and   (\ref{contJump}) at the interface points $1/3$ and $2/3$.
For the discontinuous interface point $1/9$ we use the enrichment function defined in \eqref{eqn:lin1} and for the continuous interface points $1/3$ and $2/3$ we use the enrichment function in (\ref{eqn:lin}). The test results using linear, quadratic, and cubic elements are displayed in Table \ref{tab:5} and Figure \ref{fig:5}, respectively.

%  - - - - Problem 3
\begin{figure}[h]
\centering
\begin{subfigure}{.33\textwidth}
  \centering
  \includegraphics[width=\linewidth]{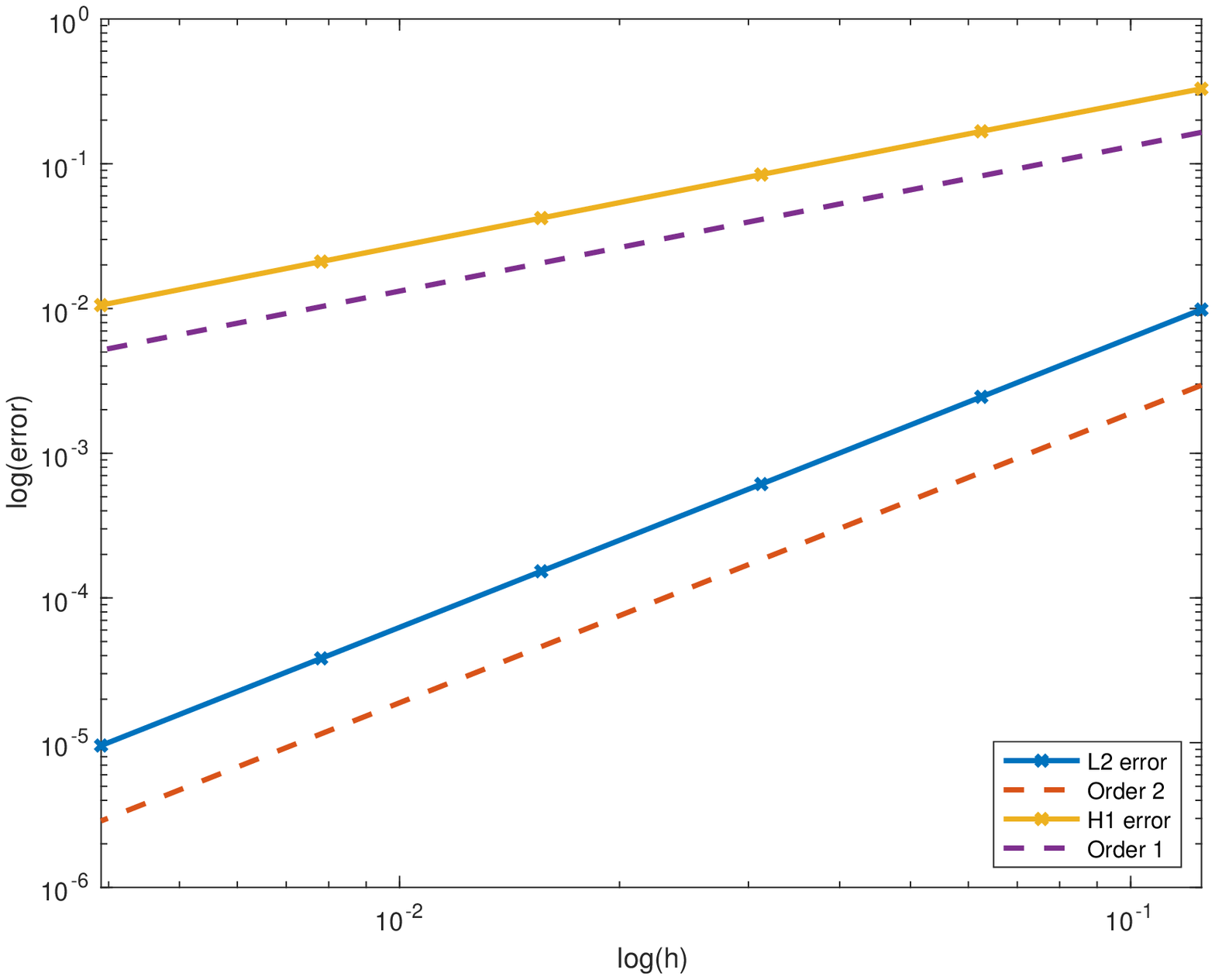}
  \caption{Linear}
  \label{fig:sub1}
\end{subfigure}%
\begin{subfigure}{.33\textwidth}
  \centering
  \includegraphics[width=\linewidth]{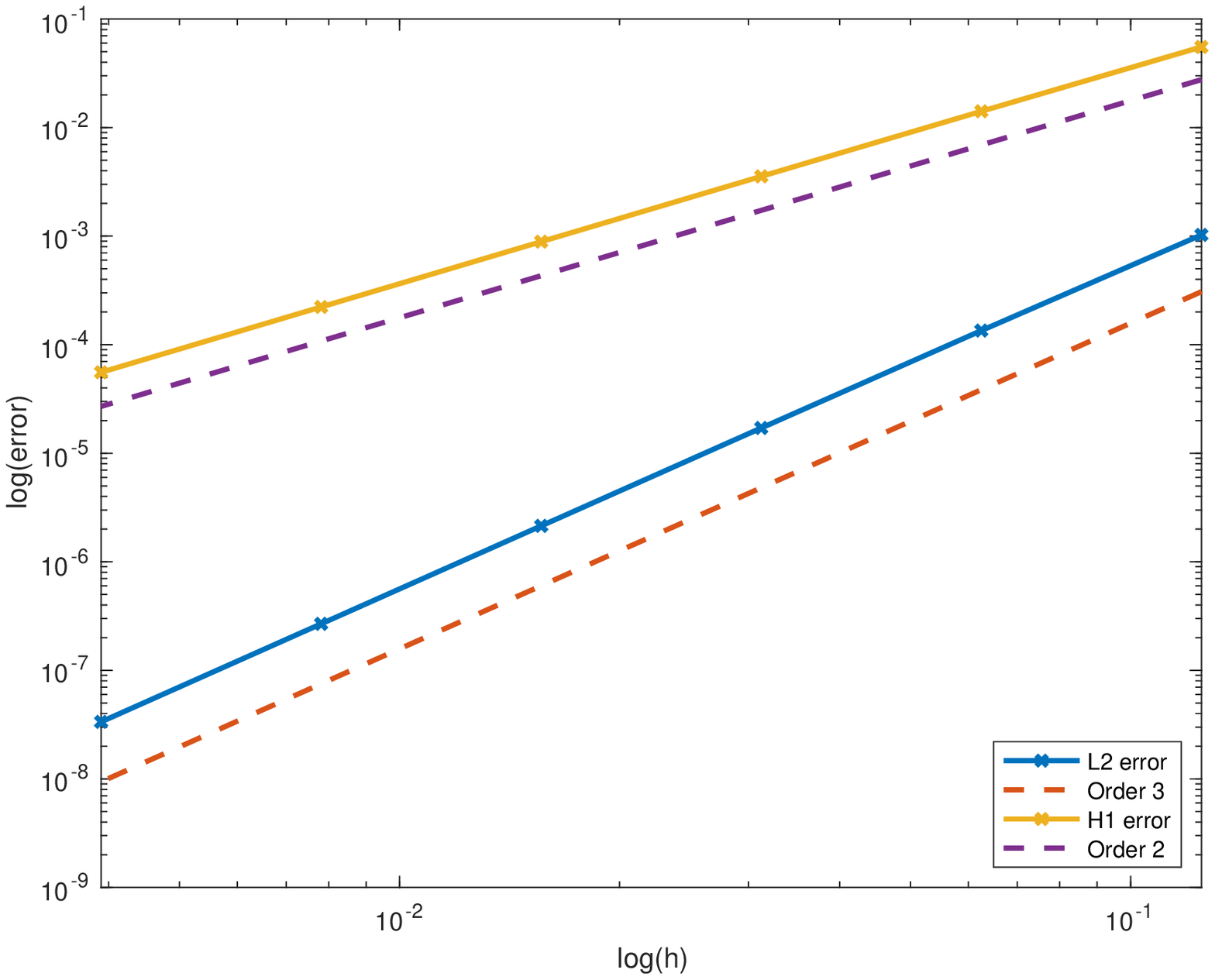}
  \caption{Quadratic}
  \label{fig:sub2}
\end{subfigure}
\begin{subfigure}{.33\textwidth}
  \centering
  \includegraphics[width=\linewidth]{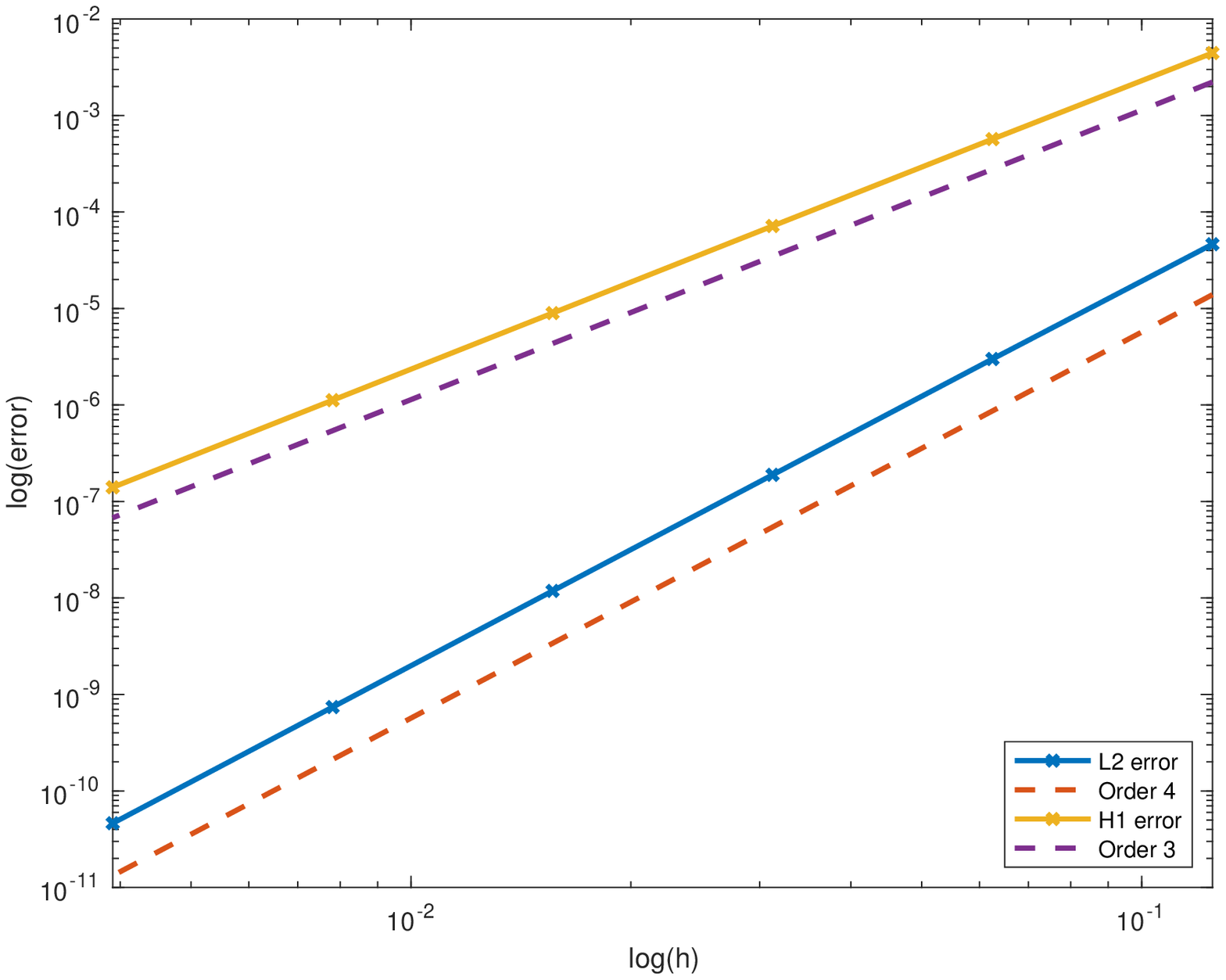}
  \caption{Cubic}
  \label{fig:sub2}
\end{subfigure}
\caption{$L_{2}$ and broken $H^{1}$ convergence rates for Problem 3 for linear, quadratic, and cubic basis functions}
\label{fig:5}
\end{figure}

\begin{table}%[!h]
\begin{center}
\begin{tabular}{l cc cc cc}
\toprule
Problem 3 & \multicolumn{2}{c}{Linear} & \multicolumn{2}{c}{Quadratic} & \multicolumn{2}{c}{Cubic} \\
\cmidrule(lr){2-3} \cmidrule(lr){4-5} \cmidrule(lr){6-7}
 & Nodal error & SCN & Nodal error & SCN & Nodal error & SCN \\
$ h=1/8 $ &	 2.92e-2 &		 2.35e4 &		 1.01e-3 &		 4.62e8 &	 	 1.16e-5 &		 6.98e7 \\
 \hline
$ h=1/16 $ &	 6.43e-3 &	 	 7.18e6 &		 5.81e-5 &	 	 3.09e7 &	 	 1.77e-7 &	 	 1.83e7 \\
 \hline
$ h=1/32 $ &	 1.53e-3 &		 2.03e7  &		 3.72e-6 &	 	 1.04e7 &	 	 2.85e-9 &	 	 1.10e7 \\
 \hline
$ h=1/64 $ &		 3.82e-4 &		 2.13e7 &		 2.31e-7 &		 9.28e6  &	 	 4.49e-11 &	 	 1.59e7 \\
 \hline
$ h=1/128 $ &		 9.51e-5 &		 1.71e7 &		 1.44e-8 &	 	 4.62e6 &	 	 9.62e-13 &	 	 8.99e7 \\
 \hline
$ h=1/256 $ &	 	 2.38e-5 &		 3.08e7 &		 9.00e-10 &		 6.79e6  &		 8.70e-13 &		 1.71e8 \\
\bottomrule
\end{tabular}
 \caption{$L^{2}$, broken $H^{1}$, nodal errors, and scaled condition numbers (SCN) with discontinuous jump conditions for Problem 3  for linear, quadratic, and cubic basis functions.}
 \label{tab:5}
\end{center}
\end{table}

\newpage

%\newpage
\section{Concluding Remarks.}
In this paper, we extend the lower-order GFEM developed recently for interface problems
with discontinuous solutions to arbitrarily high-order elements in the spatial dimension.
The main challenge has been the fact that the enrichment function constructed for linear GFEM
is not sufficient to capture the discontinuous feature of the solution to the interface problem.
\rev{Herein, we introduce two enrichment functions locally on the interval containing the interface.
With these novel enrichment functions, we generalize the method to high-order elements.}
These two enrichment functions, however, introduce linear dependence among the basis functions.
We overcome this issue by removing, from the standard FEM space,
the bubble functions that are associated with the element containing the interface.
We establish optimal error estimates for arbitrary-order elements and
demonstrate the performance of the method with various numerical examples.

\rev{This work serves as preliminary development of GFEM for a larger class of interface problems as it is focused on 1D and an extension to a special case of 2D problems}.
A natural direction of future work is the extension to multiple \rev{dimensions with general interfaces such as curved ones}.
The challenge lies in the construction of enrichment functions such that the enriched space
has the optimal approximability while keeping the condition number from fast growth with respect to mesh size.
Another future work direction is the study of the related time-dependent interface problems.
Many real application interface problems are time-dependent, for example,
the precipitating quasigeostrophic equations for climate modeling.
\rev{Herein, the GFEM provides an alternative numerical solver of high-order accuracy, stability, and robustness.}

%\color{black}

%\section{Funding and/or Conflicts of interests/Competing interests}
%%Declaration of Interest Statement: \newline
%
%The authors declare that they have no known competing financial interests or personal relationships that could have appeared to influence the work reported in this paper.

\bibliographystyle{siam}
\bibliography{HighOrder}

\begin{thebibliography}{10}

\bibitem{Ammari}
{\sc H.~Ammari, J.~Garnier, H.~Kang, M.~Lim, and S.~Yu}, {\em Generalized
  polarization tensors for shape description}, Numerische Mathematik, 126
  (2014), pp.~199--224.

\bibitem{ATTChou2021}
{\sc C.~Attanayake and S.-H. Chou}, {\em Superconvergence and flux recovery for
  an enriched finite element method.}, International Journal of Numerical
  Analysis \& Modeling, 18 (2021), pp.~656--673.

\bibitem{Babuska1}
{\sc I.~Babu{\v{s}}ka and U.~Banerjee}, {\em Stable generalized finite element
  method ({SGFEM})}, Computer methods in applied mechanics and engineering, 201
  (2012), pp.~91--111.

\bibitem{Babuska2}
{\sc I.~Babu{\v{s}}ka, U.~Banerjee, and K.~Kergrene}, {\em Strongly stable
  generalized finite element method: Application to interface problems},
  Computer Methods in Applied Mechanics and Engineering, 327 (2017),
  pp.~58--92.

\bibitem{babuvska2004generalized}
{\sc I.~Babu{\v{s}}ka, U.~Banerjee, and J.~E. Osborn}, {\em Generalized finite
  element methods—main ideas, results and perspective}, International Journal
  of Computational Methods, 1 (2004), pp.~67--103.

\bibitem{Belyt}
{\sc T.~Belytschko and L.~Black}, {\em Elastic crack growth in finite elements
  with minimal remeshing}, International Journal for Numerical Methods in
  Engineering, 45 (1999), pp.~601--620.

\bibitem{Bordas}
{\sc S.~P. Bordas, E.~Burman, M.~G. Larson, and M.~A. Olshanskii}, {\em
  Geometrically unfitted finite element methods and applications}, Lecture
  Notes in Computational Science and Engineering, 121 (2017), pp.~6--8.

\bibitem{Chou}
{\sc S.-H. Chou}, {\em An immersed linear finite element method with interface
  flux capturing recovery}, Discrete \& Continuous Dynamical Systems-B, 17
  (2012), p.~2343.

\bibitem{ChouAttan2}
{\sc S.-H. Chou and C.~Attananyake}, {\em Construction of discontinuous
  enrichment functions for enriched {FEM}'s for interface elliptic problems
  {1D}}, Journal of Computational and Applied Mathematics,  (to appear).

\bibitem{Deng}
{\sc Q.~Deng and V.~Calo}, {\em Higher order stable generalized finite element
  method for the elliptic eigenvalue and source problems with an interface in
  {1D}}, Journal of Computational and Applied Mathematics, 368 (2020),
  p.~112558.

\bibitem{Fries}
{\sc T.-P. Fries and T.~Belytschko}, {\em The extended/generalized finite
  element method: an overview of the method and its applications},
  International Journal for Numerical Methods in Engineering, 84 (2010),
  pp.~253--304.

\bibitem{guo2019group}
{\sc R.~Guo and T.~Lin}, {\em A group of immersed finite-element spaces for
  elliptic interface problems}, IMA Journal of Numerical Analysis, 39 (2019),
  pp.~482--511.

\bibitem{guo2018nonconforming}
{\sc R.~Guo, T.~Lin, and X.~Zhang}, {\em Nonconforming immersed finite element
  spaces for elliptic interface problems}, Computers \& Mathematics with
  Applications, 75 (2018), pp.~2002--2016.

\bibitem{Hahn}
{\sc D.~W. Hahn and M.~N. {\"O}zisik}, {\em Heat conduction}, John Wiley \&
  Sons, 2012.

\bibitem{he2013immersed}
{\sc X.~He, T.~Lin, Y.~Lin, and X.~Zhang}, {\em Immersed finite element methods
  for parabolic equations with moving interface}, Numerical Methods for Partial
  Differential Equations, 29 (2013), pp.~619--646.

\bibitem{jo2019recent}
{\sc G.~Jo and D.~Y. Kwak}, {\em Recent development of immersed {FEM} for
  elliptic and elastic interface problems}, Journal of the Korean Society for
  Industrial and Applied Mathematics, 23 (2019), pp.~65--92.

\bibitem{Kim}
{\sc J.~Ka{\v{c}}ur and R.~Van~Keer}, {\em A nondestructive evaluation method
  for concrete viods: frequency differential electrical impedance scanning},
  SIAM Journal on Applied Mathematics, 69 (2009), pp.~1759--1771.

\bibitem{kergrene2016stable}
{\sc K.~Kergrene, I.~Babu{\v{s}}ka, and U.~Banerjee}, {\em Stable generalized
  finite element method and associated iterative schemes; application to
  interface problems}, Computer Methods in Applied Mechanics and Engineering,
  305 (2016), pp.~1--36.

\bibitem{Kruit1}
{\sc P.~A. Krutitskii}, {\em The jump problem for the {H}elmholtz equation and
  singularities at the edges}, Applied Mathematics Letters, 13 (2000),
  pp.~71--76.

\bibitem{LI}
{\sc Z.~Li}, {\em The immersed interface method using a finite element
  formulation}, Applied Numerical Mathematics, 27 (1998), pp.~253--267.

\bibitem{LI:2006}
{\sc Z.~Li and K.~Ito}, {\em The immersed interface method: numerical solutions
  of PDEs involving interfaces and irregular domains}, SIAM, 2006.

\bibitem{Li2004}
{\sc Z.~Li, T.~Lin, Y.~Lin, and R.~C. Rogers}, {\em An immersed finite element
  space and its approximation capability}, Numerical Methods for Partial
  Differential Equations: An International Journal, 20 (2004), pp.~338--367.

\bibitem{Li2003}
{\sc Z.~Li, T.~Lin, and X.~Wu}, {\em New cartesian grid methods for interface
  problems using the finite element formulation}, Numerische Mathematik, 96
  (2003), pp.~61--98.

\bibitem{Moes}
{\sc N.~Mo{\"e}s, J.~Dolbow, and T.~Belytschko}, {\em A finite element method
  for crack growth without remeshing}, International journal for numerical
  methods in engineering, 46 (1999), pp.~131--150.

\bibitem{Pontrelli}
{\sc G.~Pontrelli and F.~de~Monte}, {\em A multi-layer wall model for coronary
  drug-eluting stents}, Int. J. Heat Mass Transf., 50 (2007), pp.~3658--3889.

\bibitem{wahlbin2006superconvergence}
{\sc L.~Wahlbin}, {\em Superconvergence in Galerkin finite element methods},
  Springer, 2006.

\bibitem{Wang}
{\sc H.~Wang, J.~Chen, P.~Sun, and F.~Qin}, {\em A conforming enriched finite
  element method for elliptic interface problems}, Applied Numerical
  Mathematics, 127 (2018), pp.~1--17.

\bibitem{Zhang1}
{\sc H.~Zhang, X.~Feng, and K.~Wang}, {\em Long time error estimates of{ IFE}
  methods for the unsteady multi-layer porous wall model}, Applied Numerical
  Mathematics, 156 (2020), pp.~303--321.

\bibitem{Zhang2}
{\sc H.~Zhang, T.~Lin, and Y.~Lin}, {\em Linear and quadratic immersed finite
  element methods for the multi-layer porous wall model for coronary
  drug-eluting stents.}, International Journal of Numerical Analysis \&
  Modeling, 15 (2018).

\bibitem{Zhang}
{\sc H.~Zhang and K.~Wang}, {\em Long-time stability and asymptotic analysis of
  the {IFE} method for the multilayer porous wall model}, Numerical Methods for
  Partial Differential Equations, 34 (2018), pp.~419--441.

\bibitem{zhang2021generalized}
{\sc J.~Zhang, Q.~Deng, and X.~Li}, {\em A generalized isogeometric analysis of
  elliptic eigenvalue and source problems with an interface}, Journal of
  Computational and Applied Mathematics,  (2021), p.~114053.

\bibitem{zhang2020stable}
{\sc Q.~Zhang and I.~Babu{\v{s}}ka}, {\em A stable generalized finite element
  method ({SGFEM}) of degree two for interface problems}, Computer Methods in
  Applied Mechanics and Engineering, 363 (2020), p.~112889.

\bibitem{zhang2019strongly}
{\sc Q.~Zhang, U.~Banerjee, and I.~Babu{\v{s}}ka}, {\em Strongly stable
  generalized finite element method ({SSGFEM}) for a non-smooth interface
  problem}, Computer Methods in Applied Mechanics and Engineering, 344 (2019),
  pp.~538--568.

\bibitem{zhang2020strongly}
\leavevmode\vrule height 2pt depth -1.6pt width 23pt, {\em Strongly stable
  generalized finite element method ({SSGFEM}) for a non-smooth interface
  problem {II}: A simplified algorithm}, Computer Methods in Applied Mechanics
  and Engineering, 363 (2020), p.~112926.

\end{thebibliography}

\appendix{} \label{appendix}

\section{Scales of Diagonal Entries}\label{scale_reason}
In this appendix, we examine the order of magnitude of the diagonal entries of the stiffness matrix for the linear element.
We focus on the linear case $p=1$. Let us calculate the diagonal entries of the stiffness matrix $\tilde A$ whose
   diagonal entries take the form of
   \begin{equation}\notag
     \tilde a_{ii}=a(\xi_i,\xi_i)=\int_I\beta\xi_i'\xi_i'dx+\frac{[\xi_i]_\alpha^2}{\lambda}:=a_{ii}+r_{ii}
     \end{equation}
with
   \[\{\xi_i\}:=   \{\phi_1,\phi_2,\ldots,\phi_{n-1},\phi_k\psi_0,\phi_{k+1}\psi_0,\phi_k\psi_1,\phi_{k+1}\psi_1\}.
   \]
    Note that the jump-related terms $r_{ii}=0$ except for $r_{jj}, j=n,n+1,n+2,n+3$.

    Assume the diffusion coefficient has only two values $\beta^-,\beta^+$ for the time being (to get the general result we can use
    $\beta_{\min}=\min_{x\in I}\{\beta^-(x),\beta^+(x)\}, \beta_{\max}=\max_{x\in I}\{\beta^-(x),\beta^+(x)\}$).
   Then
   \begin{equation*}
   a_{ii}=\int_0^1\beta \phi'_i\phi'_idx=
   \begin{cases}
   \frac{2\beta^-}{h}\quad &i=1,\ldots,k-1\\
   \frac{\beta^-}{h}+\frac{\beta^-}{h^2}(\alpha-x_k)+\frac{\beta^+}{h^2}(x_{k+1}-\alpha)\quad &i=k\\
   \frac{\beta^+}{h}+\frac{\beta^-}{h^2}(\alpha-x_k)+\frac{\beta^+}{h^2}(x_{k+1}-\alpha)\quad &i=k+1\\
   \frac{2\beta^+}{h},\quad &i=k+2,\ldots, n-1
   \end{cases}
   \end{equation*}
   and thus
   \begin{equation*}\label{overestimate1}
   \frac{2\beta_{\min}}{h}\leq a_{ii}\leq \frac{2\beta_{\max}}{h},\quad  1\leq i\leq n-1.
   \end{equation*}

   Thus, there holds
   \begin{align*}
   a_{nn}=\int_{x_k}^{x_{k+1}}\beta [(\phi_k\psi_0)']^2dx&=\beta^-m_1^2\int_{x_k}^\alpha(\phi_k'(x-x_k)+\phi_k)^2dx\\
   a_{n+1,n+1}=\int_{x_k}^{x_{k+1}}\beta [(\phi_{k+1}\psi_0)']^2dx&= \beta^-m_1^2\int_{x_k}^\alpha(\phi_{k+1}'(x-x_k)+\phi_{k+1})^2dx
  \end{align*}

   \begin{align*}
   a_{n+2,n+2}=\int_{x_k}^{x_{k+1}}\beta [(\phi_k\psi_1)']^2dx&=\beta^+m_2^2\int_\alpha^{x_{k+1}}(\phi_k'(x-x_{k+1})+\phi_k)^2dx\\
   a_{n+3,n+3}=\int_{x_k}^{x_{k+1}}\beta [(\phi_{k+1}\psi_1)']^2dx&= \beta^+m_2^2\int_\alpha^{x_{k+1}}(\phi_{k+1}'(x-x_{k+1})+\phi_{k+1})^2dx.
   \end{align*}

Note that

\begin{align*}&
\int_{x_k}^\alpha(\phi_k'(x-x_k)+\phi_k)^2dx=h_k\int_{0}^{(\alpha-x_k)/h_k}(2\hat x-1)^2d\hat x\\
\notag&=\frac{1}{6h_k^2}\left((2(\alpha-x_k)-h_k)^3+h_k^3\right)\\
\notag&=\frac{1}{6h_k^2}2(\alpha-x_k)\left((2\alpha-x_k-x_{k+1})^2-(2\alpha-x_k-x_{k+1})h_k+h_k^2\right)\\
\notag&\geq\frac{1}{6h_k^2}2(\alpha-x_k)\frac{1}{2}\left((2\alpha-x_k-x_{k+1})-h_k\right)^2\text{ by }a^2-ab+b^2\geq \frac{1}{2}(a-b)^2\\
\notag&=\frac{2}{3h_k^2}(\alpha-x_k)(\alpha-x_{k+1})^2.
\end{align*}

\begin{align*}
\label{m12}\int_{x_k}^\alpha(\phi_{k+1}'(x-x_k)+\phi_{k+1})^2dx=\frac{4}{3h_k^2}(\alpha-x_k)^3.
\end{align*}

Also, a simple calculation leads to
\begin{align*}
\int_\alpha^{x_{k+1}}(\phi_k'(x-x_{k+1})+\phi_k)^2dx&=\frac{4}{3h_k^2}(x_{k+1}-\alpha)^3
\end{align*}
\begin{align*}
&\int_\alpha^{x_{k+1}}(\phi_{k+1}'(x-x_{k+1})+\phi_{k+1})^2dx=h_k\int_{(\alpha-x_k)/h_k}^1(2\hat x-1)^2d\hat x\\
\notag&=\frac{1}{6h_k^2}\left((-2(\alpha-x_k)+h_k)^3+h_k^3\right)\\
\notag&=\frac{1}{6h_k^2}2(x_{k+1}-\alpha)\left((2\alpha-x_k-x_{k+1})^2+(2\alpha-x_k-x_{k+1})h_k+h_k^2\right)\\
\notag&\geq\frac{1}{6h_k^2}2(x_{k+1}-\alpha)\frac{1}{2}\left((2\alpha-x_k-x_{k+1})+h_k\right)^2 \, \text{ by }a^2+ab+b^2\geq \frac{1}{2}(a+b)^2\\
\notag&=\frac{2}{3h_k^2}(x_{k+1}-\alpha)(\alpha-x_{k})^2.
\end{align*}

Recall the jump-related terms $r_{ii}=0$ except for $r_{jj}, j=n,n+1,n+2,n+3$ and are of the lower order compared with $a_{ii}$ in magnitude due to the order of derivative:
\begin{equation*}
\frac{1}{\lambda}\cdot
\begin{cases}
[\phi_k\psi_0]_\alpha=m_1^2(\alpha-x_k)^2(\alpha-x_{k+1})^2h_k^{-2}\\
[\phi_{k+1}\psi_0]_\alpha=m_1^2(\alpha-x_k)^2(\alpha-x_{k})^2h_k^{-2},
\end{cases}
\end{equation*}
and
\begin{equation*}
\frac{1}{\lambda}\cdot
\begin{cases}
[\phi_k\psi_1]_\alpha=m_2^2(\alpha-x_{k+1})^2(\alpha-x_{k+1})^2h_k^{-2}\\
[\phi_{k+1}\psi_1]_\alpha=m_2^2(\alpha-x_k)^2(\alpha-x_{k+1})^2h_k^{-2}.
\end{cases}
\end{equation*}

From the above estimates, it is clear that $m_1=(\alpha-x_{k+1})/h_k,m_2=(\alpha-x_k/h_k)$ generate a system of condition numbers with a scale comparable to the system from the standard FEM.

\end{document}